\documentclass[final,3p,times,12pt]{elsarticle}
\usepackage[fleqn]{amsmath}
\usepackage{amsthm}
\usepackage{amssymb}
\usepackage{graphicx}
\usepackage{subfigure}
\usepackage{color}
\usepackage{array}
\usepackage{mathrsfs}
\usepackage{float}
\usepackage{epstopdf}
\usepackage{booktabs}
\usepackage{tabularx}
\usepackage{lineno}
\usepackage{grffile}
\usepackage{extarrows}
\usepackage{microtype}
\usepackage{ragged2e}
\allowdisplaybreaks[4]
\graphicspath{{Fig/}}
\usepackage[colorlinks,linkcolor=blue,anchorcolor=blue,citecolor=blue,CJKbookmarks=True]{hyperref}

\newcommand{\E}{\mathbb{E}}
\newcommand{\PP}{\mathbb{P}}
\newcommand{\RR}{\mathbb{R}}

\newcommand{\tabincell}[2]{\begin{tabular}{@{}#1@{}}#2\end{tabular}}
\def\[{\langle} \def\]{\rangle}
\newdefinition{lemma}{Lemma}[section]
\newdefinition{theorem}{Theorem}[section]
\newdefinition{definition}{Definition}[section]
\newdefinition{assumption}{Assumption}[section]
\newdefinition{proposition}{Proposition}[section]
\newdefinition{remark}{Remark}[section]
\newdefinition{claim}{Claim}[section]
\newdefinition{corollary}{Corollary}[section]
\newdefinition{example}{Example}[section]

\begin{document}
	\newpage
	\begin{frontmatter}
		\setcounter{page}{1}
		\title{Strong convergence rate of the positivity-preserving logarithmic truncated EM method for multi-dimensional stochastic differential equations with positive solutions$^{\star}$}
		\author{Xingwei Hu$^{a}$,~~Xinjie Dai$^b$,~~Aiguo Xiao$^{a,*}$}
		%~~Mengjie Wang$^a$,~~Xinjie Dai$^b$,~~Yanyan Yu$^a$
		\footnote{\textsuperscript{$\star$} This work is supported by National Natural Science Foundation of China (Nos.12471391 and 12401547) and the Postgraduate Innovation Foundations of Hunan Province (No. CX20250930) and Xiangtan University (No. XDCX2025Y185). X. Hu is supported by the China Scholarship Council under 202508430056.\\
		 \indent \indent \textsuperscript{$*$} Corresponding author (Aiguo Xiao).\\
			\emph{Email addresses}: \texttt{xingweihu@smail.xtu.edu.cn} (X.\ Hu), %\texttt{mengjiewang@smail.xtu.edu.cn} (M.\ Wang), 
			\texttt{dxj@ynu.edu.cn} (X.\ Dai),
			%\texttt{yanyanyu@smail.xtu.edu.cn} (Y.\ Yu), 
			\texttt{xag@xtu.edu.cn} (A.\ Xiao).}
		\address{$^a$School of Mathematics and Computational Science $\&$ Hunan Key Laboratory for Computation and Simulation in Science and Engineering, Xiangtan University, Xiangtan, Hunan 411105, China \\
			$^b$School of Mathematics and Statistics, Yunnan University, Kunming, Yunnan 650504, China}
		\date{}

		\begin{abstract}
			\par
			As a combination of the logarithmic transformation with the truncated Euler--Maruyama (TEM) scheme, the positivity-preserving logarithmic truncated Euler--Maruyama (LTEM) scheme has been generally developed for scalar stochastic differential equations (SDEs) with positive solutions. A subsequent question arises: can this method be extended to effectively solve general multi-dimensional SDEs with positive solutions? The answer to this question is affirmative. In this paper, we construct the positivity-preserving LTEM scheme to solve this type of system and demonstrate its suboptimal strong convergence rate of this scheme. On the other hand, when the underlying system degenerates into a scalar equation, the latest LTEM scheme analyzed by Tang \& Mao (2024) is applicable to scalar SDEs with weak conditions, but its strong convergence rate is suboptimal. Based on this, we will theoretically demonstrate the optimal convergence rate of the LTEM method without infinitesimal factors in the scalar case. The proof strategy exactly improves its convergence rate from suboptimal to optimal. Finally, Numerical examples are provided to validate the effectiveness and positivity-preserving of the LTEM method.
		\end{abstract}
		
	\end{frontmatter}
	
	\section{Introduction}
	\label{sec.1}
	\indent In 2015, Mao \cite{M15} introduced the TEM method for multi-dimensional nonlinear SDEs and set up strong convergence findings without specifying the corresponding convergence rate. In 2016, Mao \cite{M16} delved deeper into its convergence rate and demonstrated that it exhibited a suboptimal convergence rate under certain additional conditions. The combination of this method and logarithmic transformation gives rise to the positivity-preserving LTEM method, which has been developed and analyzed in \cite{LG2023,TX2024,PPlogTM} for scalar SDEs with positive solutions. Specifically, its strong convergence rates are suboptimal and optimal in \cite{PPlogTM} and \cite{LG2023}, respectively. A novel first-order positivity-preserving methods is further proposed in \cite{HXW}. Tang \& Mao \cite{TX2024} conducted deeper research on the LTEM method under weaker conditions, revealing its suboptimal strong convergence rate. Moreover, several positivity-preserving methods \cite{CJI2016,LNX2025,NS2014} are developed using the Lamperti transformation and modified EM methods for the scalar SDEs with positive solutions. \\
	\indent All the aforementioned literature is excellent; however, it is only limited to the scalar case. A natural follow-up question is whether the LTEM method can be extended to solve general multi-dimensional SDEs with positive solutions. Notably, when dealing with multi-dimensional systems, using Lamperti or logarithmic transformations may render the general monotonicity condition inadequate for the transformed SDEs, such that analyzing the convergence rate becomes a challenge when transformations are applied. Nevertheless, our answer to the question above is affirmative. Therefore, our primary goal of this paper is to extend the LTEM method to multi-dimensional scenarios and demonstrate its suboptimal strong convergence rate.\\
	\indent However, for multi-dimensional positivity-preserving schemes, there exist some results.
	For examples, for stochastic Lotka-Volterra (LV) competition model, the novel positivity-preserving TEM method and the positivity-preserving Lamperti-type EM scheme with the optimal strong convergence are proposed by  Mao, Wei \& Wiriyakraikul \cite{MF21} and Li \& Cao \cite{APPL2023}, respectively. Besides, Cai, Guo \& Mao \cite{YQX2024} proposed a positivity-preserving TEM method for multi-dimensional superlinear
	SDEs with positive solutions. Lastly, Cai, Mao \& Fei \cite{YXF2024} demonstrated that an exponential EM scheme exhibits the suboptimal strong convergence rate for multi-dimensional stochastic Kolmogorov equations. Hu, Dai \& Xiao \cite{arxivhu} proposed a positivity-preserving TEM method and demonstrated its optimal strong convergence rate for general stochastic systems yielding positive solutions. \\
	\indent On the other hand, another goal of this paper is to demonstrate the optimal strong convergence rate of the LTEM method when the dimension $d$ of \eqref{sde} equals to $1$. In \cite{TX2024}, the expressions $\sup_{ s \in[0,T]}\E[|y_\Delta(s)-\bar{y}_\Delta(s)|^p]$ and $\sup_{ s \in[0,T]}\E\big[\big|\frac{y_\Delta(s)}{\bar{y}_\Delta(s)}-1\big|^p\big]$ are estimated by $C\Delta^\frac{p}{2}(\eta(\Delta))^p$. The employment of these estimations in the convergence theory yields the suboptimal strong convergence rate. To theoretically achieve the optimal convergence rate, we need to re-evaluate these expressions. A detailed process involves estimating that $|\lambda_\Delta(\bar{z}(v))|$ and $|\sigma_\Delta(\bar{z}(v))|$ ($v\in[0,T]$) less than or equal to $C$ using certain assumptions along with the moment bound of the numerical solutions, instead of $\eta(\Delta)$. Based on this, one can successfully derive those estimations as $C\Delta^{\frac{p}{2}}$, rather than $ C\Delta^{\frac{p}{2}}(\eta(\Delta))^p$. This proof strategy effectively enhances the convergence rate. \\
	\indent This work makes two main contributions:
	\begin{itemize}
		\item We study the positivity-preserving LTEM method for approximating SDE (1). The LTEM method achieves the suboptimal strong convergence rate. Numerical examples verify the positivity-preserving properties and effectiveness of the method.
		\item The latest LTEM method \cite{TX2024} exhibits the suboptimal convergence rate. We optimize the strong convergence rate of the LTEM method for the scalar SDEs with positive solutions, improving it from suboptimal to optimal.  
	\end{itemize}
	
	\indent This paper is organized as follows. In Section 2, we outline some notations, introduce several important lemmas and construct the LTEM method. In the next section, we demonstrate the strong convergence analysis of the LTEM method. In Section 4, we show the optimal strong convergence rate of the LTEM method in the scalar case. In Section 5, we exhibit various numerical experiments in support of our theoretical conclusions. Finally, we briefly conclude our work.
	
	\section{Multi-dimensional case and the LTEM method}
	\label{sec.2}
	
	\subsection{Notations and important lemmas}
	\par $G^T$ stands for transposition of a vector or matrix $G$. Define $\mathbb{E}$ as the expectation with respect to $\PP$. The positive cone in $\RR^d$ is denoted by $\RR_+^d$, which is defined as $\RR_+^d = \{ y \in \RR^d : y_i > 0 \text{ for all} 1 \leq i \leq d \}$. For any set $A$, its indicator function is denoted by $I_A$, defined as $I_A(x) = 1$ if $x \in A$ and $0$ otherwise. If $B$ is a matrix, we define its trace norm as $|B| = \sqrt{\text{trace}(B^T B)}$. For a vector $x \in \RR^d$, the notation $|x|$ refers to the Euclidean norm. Set $m \wedge n=\min\left\{m,n\right\}$ and $m \vee n=\max\left\{m,n\right\}$, where $m$ and $n$ are real numbers. For any vectors $U=(U_1,U_2,\cdots,U_d),V=(V_1,V_2,\cdots,V_d)\in\RR^d$, $UV=(U_1V_1,U_2,V_2,\cdots,U_dV_d)$ and $\frac{U}{V}=(\frac{U_1}{V_1},\frac{U_2}{V_2},\cdots,\frac{U_d}{V_d})$. We denote $C$ as a positive constant independent of $\Delta$ (step size, see in Section 2.2) that can vary in different contexts. \\
	\indent Consider a $d$-dimensional SDE with positive solutions
	\begin{equation}\label{sde}
		\mathrm{d}y(t)=\lambda(y(t))\mathrm{d}t+\sigma(y(t))\mathrm{d}B(t),\quad 0<t\leq T, \quad y(0)=y_0\in \RR^d_+,
	\end{equation}
	where $\lambda=(\lambda^1,\lambda^2,\cdots,\lambda^d)^T:\RR_+^d\rightarrow \RR^d$ and $\sigma=(\sigma^{i,j})_{d\times m}=(\sigma^1,\sigma^2,\cdots,\sigma^m)=(\sigma^T_1,\sigma^T_2,\cdots,\sigma^T_d)^T:\RR_+^d\rightarrow \RR^{d\times m}$.\\
	\indent We posit the subsequent hypothesis to ensure that the multi-dimensional SDE \eqref{sde} admits a unique global solution with values in $\mathbb{R}^d_+$. 
	\begin{assumption}\label{as.1}
		The coefficients $\lambda$ and $\sigma$ of \eqref{sde} satisfy the non-globally Lipschitz condition: there exist certain positive constants $L_1$, $\alpha$ and $\beta$ such that the inequality
		\begin{flalign*}
			|\lambda(\tilde{y})-\lambda(\hat{y})|\vee|\sigma(\tilde{y})-\sigma(\hat{y})|\leq L_1(1+|\tilde{y}|^\alpha+|\hat{y}|^\alpha+|\tilde{y}|^{-\beta}+|\hat{y}|^{-\beta})|\tilde{y}-\hat{y}|
		\end{flalign*} 
		holds for all $\tilde{y},\hat{y}\in \RR^d_+$. Besides, there exist some positive constants $y_i^*,H^i,K$ along with $J>1$ such that for any $\check{y}=(\check{y}_1,\check{y}_2,\cdots,\check{y}_d)^T\in\RR^d_+$ and any $i\in\{1,2,\cdots,d\}$,
		\begin{flalign*}
			\left\{ 
			\begin{aligned}
				&\check{y}_i\lambda^i(\check{y})-\frac{K+1}{2}|\sigma_i(\check{y})|^2\geq0, &\check{y}_i&\in(0,y_i^*);\\
				&\check{y}_i\lambda^i(\check{y})+\frac{J-1}{2}|\sigma_i(\check{y})|^2\leq H^i(1+\check{y}_i^2), &\check{y}_i&\in[y_i^*,\infty).
			\end{aligned}
			\right.
		\end{flalign*}
	\end{assumption}
	
	\begin{remark}\label{Remark1}
		As noted in Remark 2.1 in \cite{arxivhu}, we may deduce from Assumption \ref{as.1} that
		\begin{flalign*}
			|\lambda(y)|\vee |\sigma(y)|\leq C(1+|y|^{\alpha+1}+|y|^{-\beta}),\quad \forall y\in\RR^d_+.
		\end{flalign*}
	\end{remark}
	\begin{lemma}(Lemma 2.1 in \cite{arxivhu})\label{Lm2.1}
		Let Assumption \ref{as.1} hold with the parameters satisfying $J\geq2(\alpha+1)$ and $K\geq2\beta$. Then SDE \eqref{sde} admits a unique global positive solution, i.e.,  
		\begin{flalign*}
			\PP(y(t)\in\RR^d_+, \quad\forall t\in[0,T])=1.
		\end{flalign*}
		Besides, 
		\begin{flalign*}
			\sup_{ t \in[0,T]}\E[|y(t)|^{J}]< \infty \quad \text{and} \quad \sup_{ t \in[0,T]}\E[|y(t)|^{-K}]< \infty.
		\end{flalign*}
	\end{lemma}
	
	\par To construct the LTEM method for SDE \eqref{sde}. Firstly, we consider $y_i(t)$, where $y_i(t)$ represents the element associated with the $i$-th index
	\begin{equation*}
		\mathrm{d}y_i(t)=\lambda^i(y(t))\mathrm{d}t+\sum_{j=1}^{m}\sigma^{i,j}(y(t))\mathrm{d}B_j(t).
	\end{equation*}
	Then for $1 \le i \le d$, by applying a logarithmic transformation $z_i(t) = \ln(y_i(t))$, one can obtain $y_i(t) = e^{z_i(t)}$. Integrating this with the It\^o formula yields its corresponding transformed SDE. 
	\begin{flalign*}
		\mathrm{d}z_i(t)=\Big(\frac{\lambda^i(e^{z(t)})}{e^{z_i(t)}}-\frac{1}{2}\sum_{j=1}^{m}\frac{|\sigma^{i,j}(e^{z(t)})|^2}{e^{2z_i(t)}}\Big)\mathrm{d}t+\sum_{j=1}^{m}\frac{\sigma^{i,j}(e^{z(t)})}{e^{z_i(t)}}\mathrm{d}B_j(t),\quad 1\leq i\leq d.
	\end{flalign*}
	Write its matrix formulation 
	\begin{equation}\label{tsde}
		\mathrm{d}z(t)=\tilde{\lambda}(z(t))\mathrm{d}t+\tilde{\sigma}(z(t))\mathrm{d}B(t).
	\end{equation}
	Here
	\begin{equation}\label{FGDD}
		\tilde{\lambda}(z)=e^{-z}\lambda(e^z)-\frac{1}{2}e^{-2z}|\sigma(e^z)|^2\quad \text{and}\quad \tilde{\sigma}(z)=e^{-z}\sigma(e^z)
	\end{equation}
	for $z\in\RR^d$, where $z(t)=(z_1(t),z_2(t),\cdots,z_d(t))^T$, $e^{z(t)}:=(e^{z_1(t)},e^{z_2(t)},\cdots,e^{z_d(t)})^T$ and $z(0)=\ln(y(0))=(\ln (y_1(0)),\ln (y_2(0)), \cdots,\ln (y_d(0)))^T$.
	\begin{remark}\label{tau*}
		Define an arbitrary stopping time $\rho^*_n$. By using Lemma \ref{Lm2.1}, one can see 
		\begin{flalign*}
			\sup_{ t \in[0,T]}\E[|y(t\wedge\rho^*_n)|^{J}]<\infty \quad \text{and} \quad \sup_{ t \in[0,T]}\E[|y(t\wedge\rho^*_n)|^{-K}]<\infty.
		\end{flalign*}
		Fixing  $\rho^*_n=\inf\{t\in[0,T]:|{z}(t)|\geq n\}$, we then deduce
		\begin{equation*}
			e^{(J\wedge K)n} \PP(\rho_n^*\leq T)=\E[(|y(\rho_n^*)|^{J}+|y(\rho_n^*)|^{-K})I_{\{\rho_n^*\leq T\}}]\leq\E[|y(T\wedge\rho^*_n)|^{J}+|y(T\wedge\rho^*_n)|^{-K}]\leq C.
		\end{equation*}	
		Therefore, we get
		\begin{flalign*}
			\PP(\rho_n^* \leq T)\leq \frac{C}{e^{(J\wedge K)n}}.
		\end{flalign*}
	\end{remark}
	
	\subsection{The LTEM method}
	\indent Now we construct our LTEM method for solving the SDE \eqref{sde}. To begin with, we select a strictly increasing continuous function $\psi:\RR_+\rightarrow\RR_+$, which satisfies $\psi(v)\rightarrow\infty$ as $v\rightarrow\infty$ along with
	\begin{equation*}
		\sup_{|z|\leq v} \big(|\tilde{\lambda}(z)| \vee |\tilde{\sigma}(z)|^2 \big) \leq \psi(v), \quad \forall v > 0.
	\end{equation*}
	Then we defined $\psi^{-1}$ as the inverse function of $\psi$, which has the property that $(\psi(0), \infty) \rightarrow (0,\infty)$ and is also increasing. Besides, we choose a strictly decreasing function $\eta:(0,1]\rightarrow (0, \infty)$ satisfying the following property 
	\begin{equation}\label{hD}
		\lim_{\Delta \rightarrow 0}\eta(\Delta) =\infty \quad \mbox{and} \quad  \Delta^{\frac{1}{2}} \eta(\Delta) \leq M_0,
	\end{equation}
	where $M_0\geq \psi(0)\vee1$. Fix $\Delta \in (0,1]$, let $\tilde{\lambda}_\Delta(x)$ and $\tilde{\sigma}_\Delta(x)$, referred as truncated functions, be defined as follows
	\begin{flalign*}
		\tilde{\lambda}_\Delta(z):=\left\{ 
		\begin{aligned}
			&\tilde{\lambda}\Big((|z|\wedge\psi^{-1}(\eta(\Delta)))\frac{z}{|z|}\Big), &z&\in\RR^d\setminus \{0\};\\
			&0, &z&=0
		\end{aligned}
		\right .
	\end{flalign*}
	and
	\begin{flalign*}
		\tilde{\sigma}_\Delta(z):=\left\{ 
		\begin{aligned}
			&\tilde{\sigma}\Big((|z|\wedge\psi^{-1}(\eta(\Delta)))\frac{z}{|z|}\Big), &z&\in\RR^d\setminus \{0\};\\
			&0, &z&=0. 
		\end{aligned}
		\right .
	\end{flalign*} 
	Clearly,
	\begin{equation}\label{bbound}
		|\tilde{\lambda}_\Delta(z)|\vee|\tilde{\sigma}_\Delta(z)|^2\leq \psi(\psi^{-1}(\eta(\Delta)))=\eta(\Delta).
	\end{equation}
	
	\indent We define a uniform mesh $\mathcal{T}_N: 0=t_0<t_1<\cdots<t_N=T$ with $t_k=k\Delta$, where $\Delta=\frac{T}{N}$ for $N\in \mathbb{N}^+$, where $\mathbb{N}^+$ denotes the ensemble of positive integers. Then the LTEM method generates numerical solutions $z_\Delta(t_k)$ to approximate $z(t_k)$ for $t_k=k\Delta$ (any given $\Delta\in(0,1]$), created by $z_\Delta(0) = z_0$ for $ k=0,1,\cdots,N-1$,
	%We utilize the same uniform mesh, denoted as $\mathcal{T}_N$, as described in Section 3. 
	\begin{flalign}\label{TEM}
		z_\Delta(t_{k+1})=z_\Delta(t_k)+\tilde{\lambda}_\Delta(z_\Delta(t_k))\Delta+\tilde{\sigma}_\Delta(z_\Delta(t_k))\Delta B_k,
	\end{flalign}
	where $\Delta B_k= B(t_{k+1})-B(t_k)$. The continuous form of \eqref{TEM} is defined as
	\begin{equation}\label{anoy}
		z_\Delta(t)=z_0+\int_{0}^{t}\tilde{\lambda}_\Delta(\bar{z}(s))\mathrm{d}s+\int_{0}^{t}\tilde{\sigma}_\Delta(\bar{z}(s))\mathrm{d}B(s),
	\end{equation}
	where $\bar{z}_\Delta(t)=z_\Delta(t_k)$ for $t \in [t_k,t_{k+1})$. Finally, the numerical solutions for the original SDE \eqref{sde} are defined as follows: 
	\begin{equation}\label{bxx}
		y_\Delta(t)=e^{z_\Delta(t)}
	\end{equation}
	for $t\in[0,T]$. The so-called LTEM method is combined \eqref{anoy} with \eqref{bxx}. 
	
	\begin{lemma}\label{multi}
		Suppose that $Q\sim N(0,\sqrt{\Delta}I_m)$ is an $m$-dimensional normal random variable, where $I_m$ is an $m$-order Identity matrix. Then for a real number $\gamma>0$, we can obtain
		\begin{flalign*}
			\E[e^{\gamma|Q|}]\leq2^me^{\frac{\gamma^2\Delta}{2}}.
		\end{flalign*}
	\end{lemma}
	\begin{proof}
		Since $Q\sim N(0,\sqrt{\Delta}I_m)$ is an $m$-dimensional normal random variable, we have
		\begin{flalign*}
			\E[e^{\gamma|Q|}]=&\int_{\RR^m}e^{\gamma|v|}\frac{1}{(2\pi\Delta)^{\frac{m}{2}}}e^{-\frac{|v|^2}{2\Delta}}\mathrm{d}v\\
			=&\frac{2^m}{(2\pi\Delta)^{\frac{m}{2}}}\int_{{[0,\infty)}^m}e^{\gamma|v|}e^{-\frac{|v|^2}{2\Delta}}\mathrm{d}v\\
			=&\Big(\frac{2}{\pi\Delta}\Big)^{\frac{m}{2}}e^{\frac{\gamma^2\Delta}{2}}\int_{{[0,\infty)}^m}e^{-\frac{|v-\gamma\Delta|^2}{2\Delta}}\mathrm{d}v\\
			=&\Big(\frac{2}{\pi}\Big)^{\frac{m}{2}}e^{\frac{\gamma^2\Delta}{2}}\int_{{[-\gamma\sqrt{\Delta},\infty)}^m}e^{-\frac{|u|^2}{2}}\mathrm{d}u\\
			\leq&\Big(\frac{2}{\pi}\Big)^{\frac{m}{2}}e^{\frac{\gamma^2\Delta}{2}}\int_{(-\infty,\infty)^m}e^{-\frac{|u|^2}{2}}\mathrm{d}u\\
			=&\Big(\frac{2}{\pi}\Big)^{\frac{m}{2}}e^{\frac{\gamma^2\Delta}{2}}\Big(2\int_{0}^{\infty}e^{-\frac{|z|^2}{2}}\mathrm{d}z\Big)^m\\
			\leq&\Big(\frac{2}{\pi}\Big)^{\frac{m}{2}}e^{\frac{\gamma^2\Delta}{2}}(2\pi)^{\frac{m}{2}}\\
			\leq&2^me^{\frac{\gamma^2\Delta}{2}}.
		\end{flalign*}
		We complete the proof.
	\end{proof}
	
	\indent Next, we will prove several properties of the numerical solutions.
	\begin{lemma}\label{LM3.1}
		For any real numbers $\bar{p}$ and $\bar{q}$, we get
		\begin{equation}\label{supsup}
			\sup_{\Delta\in (0,1]}\sup_{ t \in[0,T]}\E\Big[\Big|\frac{y_\Delta(t)}{\bar{y}_\Delta(t)}\Big|^{\bar{p}}\Big]\leq C_{\bar{p}}\quad\text{and}\quad \sup_{\Delta\in (0,1]}\sup_{ t \in[0,T]}\E\Big[\Big|\frac{\bar{y}_\Delta(t)}{y_\Delta(t)}\Big|^{\bar{q}}\Big]\leq C_{\bar{q}},
		\end{equation}
		where $C_{\bar{p}}$ and $C_{\bar{q}}$ are dependent on $\bar{p}$ and $\bar{q}$, respectively.
	\end{lemma}
	\begin{proof}
		For any given $\Delta \in (0,1]$ and $0\leq t\leq T$, there exists a unique integer $k \geq 0$ such that $t_k \leq t < t_{k+1}$, thus we obtain from \eqref{anoy} and \eqref{bxx} that
		\begin{flalign}\label{12}
			y_\Delta(t)=\bar{y}_\Delta(t)e^{\tilde{\lambda}_\Delta(\bar{z}_\Delta(t))(t-t_k)+\tilde{\sigma}_\Delta(\bar{z}_\Delta(t))(B(t)-B(t_k))}.
		\end{flalign}
		Then by \eqref{hD}, \eqref{bbound} and Lemma \ref{multi}, we obtain
		\begin{flalign*}
			\E\Big[\Big|\frac{y_\Delta(t)}{\bar{y}_\Delta(t)}\Big|^{\bar{p}}\Big]=&\E e^{\bar{p}|\tilde{\lambda}_\Delta(\bar{z}_\Delta(t))(t-t_k)+\tilde{\sigma}_\Delta(\bar{z}_\Delta(t))(B(t)-B(t_k))|}\\
			\leq&\E e^{|\bar{p}|\eta(\Delta)\Delta+(\eta(\Delta))^{\frac{1}{2}}|\bar{p}||B(t)-B(t_k)|}\\
			\leq&2^me^{|\bar{p}|(\eta(\Delta))\Delta+\frac{\bar{p}^2\eta(\Delta)\Delta}{2}}\leq C_{\bar{p}},
		\end{flalign*}
		where $C_{\bar{p}}$ is dependent on $\bar{p}$. We rewrite \eqref{12} as the following equation
		\begin{flalign*}
			\frac{\bar{y}_\Delta(t)}{y_\Delta(t)}=e^{-(\tilde{\lambda}_\Delta(\bar{z}_\Delta(t))(t-t_k)+\tilde{\sigma}_\Delta(\bar{z}_\Delta(t))(B(t)-B(t_k)))}.
		\end{flalign*}
		Similarly, we can also derive
		\begin{flalign*}
			\E\Big[\Big|\frac{\bar{y}_\Delta(t)}{y_\Delta(t)}\Big|^{\bar{q}}\Big]\leq C_{\bar{q}},
		\end{flalign*}
		where $C_{\bar{q}}$ is dependent on $\bar{q}$.
	\end{proof}
	\begin{lemma}\label{pp numerical integral}
		Suppose Assumption \ref{as.1} holds, with its parameters satisfying $J\geq2(\alpha+1)$ and $K\geq2\beta$. Let $n>1$ be a positive number and define the stopping time $\rho_n=\inf\{t\in[0,T]:|{z}_\Delta(t)|\geq n\}$. Let $\Delta\in(0,1]$ be sufficiently small such that $\psi^{-1}(\eta(\Delta))\geq n$. Then the following holds that
		\begin{flalign}\label{numerical int}
			\sup_{\Delta\in (0,\Delta^*]}\sup_{ t \in[0,T]}\E[|y_\Delta(t)|^{J}]\leq C \quad \text{and} \quad \sup_{\Delta\in (0,\Delta^*]}\sup_{ t \in[0,T]}\E[|y_\Delta(t)|^{-K}]\leq C.
		\end{flalign}
	\end{lemma}
	\begin{proof}
		One can see that $|\bar{z}_\Delta(s)|\leq n$ for $s\in[t\wedge\rho_n]$. Since $\psi^{-1}(\eta(\Delta))\geq n$, we have $|\bar{z}_\Delta(s)|\leq\psi^{-1}(\eta(\Delta))$. Hence, we derive $\tilde{\lambda}_\Delta(\bar{z}_\Delta(s))=\tilde{\lambda}(\bar{z}_\Delta(s))$ and $ \tilde{\sigma}_\Delta(\bar{z}_\Delta(s))=\tilde{\sigma}(\bar{z}_\Delta(s))$ for $s\in[0, t\wedge\rho_n]$. It follows from \eqref{anoy} that
		\begin{equation*}
			z_\Delta(t\wedge\rho_n)=z_0+\int_{0}^{t\wedge\rho_n}\tilde{\lambda}(\bar{z}(s))\mathrm{d}s+\int_{0}^{t\wedge\rho_n}\tilde{\sigma}(\bar{z}(s))\mathrm{d}B(s).
		\end{equation*}
		Using the It\^o formula for $U(t)=|e^{z_\Delta(t)}|^{\bar{p}}+|e^{z_\Delta(t)}|^{-\bar{q}}$, we can drive 
		\begin{flalign*}
			%&\E[|y_\Delta(t\wedge\rho_n)|^{J}+|y_\Delta(t\wedge\rho_n)|^{-K}]\\
			%=
			&\E[|e^{z_\Delta(t\wedge\rho_n)}|^{J}+|e^{z_\Delta(t\wedge\rho_n)}|^{-K}]\\
			=&|e^{z_0}|^{J}+|e^{z_0}|^{-K}+J\E\int_{0}^{t\wedge\rho_n}\Big(|e^{z_\Delta(s)}|^{J-2}\sum_{i=1}^{d}e^{2z_{\Delta,i}(s)}\tilde{\lambda}^i(\bar{z}_\Delta(s))\\
			&+|e^{z_\Delta(s)}|^{J-2}\sum_{i=1}^{d}e^{2z_{\Delta,i}(s)}\sum_{j=1}^{m}(\tilde{\sigma}^{i,j}(\bar{z}(s)))^2+\frac{p-2}{2}|e^{z_\Delta(s)}|^{J-4}\sum_{j=1}^{m}\big(\sum_{i=1}^{d}e^{2z_{\Delta,i}(s)}\tilde{\sigma}^{i,j}(s)\big)^2\Big)\mathrm{d}s\\
			&-K\E\int_{0}^{t\wedge\rho_n}\Big(|e^{z_\Delta(s)}|^{-K-2}\sum_{i=1}^{d}e^{2z_{\Delta,i}(s)}\tilde{\lambda}^i(\bar{z}_\Delta(s))+|e^{z_\Delta(s)}|^{-K-2}\sum_{i=1}^{d}e^{2z_{\Delta,i}(s)}\sum_{j=1}^{m}(\tilde{\sigma}^{i,j}(\bar{z}(s)))^2\\
			&-\frac{K+2}{2}|e^{z_\Delta(s)}|^{-K-4}\sum_{j=1}^{m}\big(\sum_{i=1}^{d}e^{2z_{\Delta,i}(s)}\tilde{\sigma}^{i,j}(s)\big)^2\Big)\mathrm{d}s\\
			=&|y_0|^{J}+|y_0|^{-K}+J\E\int_{0}^{t\wedge\rho_n}\Big(|y_\Delta(s)|^{J-2}\sum_{i=1}^{d}y^2_{\Delta,i}(s)\big(\frac{\lambda^i(\bar{y}_\Delta(s))}{\bar{y}_{\Delta,i}(s)}-\frac{1}{2}\sum_{j=1}^{m}\frac{|\sigma^{i,j}(\bar{y}_\Delta(s))|^2}{\bar{y}^2_{\Delta,i}(s)}\big)\\
			&+|y_\Delta(s)|^{J-2}\sum_{i=1}^{d}y^{2}_{\Delta,i}(s)\sum_{j=1}^{m}\big(\frac{\sigma^{i,j}(\bar{y}_\Delta(s))}{\bar{y}_{\Delta,i}(s)}\big)^2+\frac{J-2}{2}|y_\Delta(s)|^{J-4}\sum_{j=1}^{m}\big(\sum_{i=1}^{d}y^{2}_{\Delta,i}(s)\frac{\sigma^{i,j}(\bar{y}_\Delta(s))}{\bar{y}_{\Delta,i}(s)}\big)^2\Big)\mathrm{d}s   \\
			&-K\E\int_{0}^{t\wedge\rho_n}\Big(|y_\Delta(s)|^{-K-2}\sum_{i=1}^{d}y^2_{\Delta,i}(s)\big(\frac{\lambda^i(\bar{y}_\Delta(s))}{\bar{y}_{\Delta,i}(s)}-\frac{1}{2}\sum_{j=1}^{m}\frac{|\sigma^{i,j}(\bar{y}_\Delta(s))|^2}{\bar{y}^2_{\Delta,i}(s)}\big)\\
			&+|y_\Delta(s)|^{-K-2}\sum_{i=1}^{d}y^{2}_{\Delta,i}(s)\sum_{j=1}^{m}\big(\frac{\sigma^{i,j}(\bar{y}_\Delta(s))}{\bar{y}_{\Delta,i}(s)}\big)^2-\frac{K+2}{2}|y_\Delta(s)|^{-K-4}\sum_{j=1}^{m}\big(\sum_{i=1}^{d}y^{2}_{\Delta,i}(s)\frac{\sigma^{i,j}(\bar{y}_\Delta(s))}{\bar{y}_{\Delta,i}(s)}\big)^2\Big)\mathrm{d}s   \\
			=&|y_0|^{J}+|y_0|^{-K}+J\E\int_{0}^{t\wedge\rho_n}\Big(|y_\Delta(s)|^{J-2}\sum_{i=1}^{d}\big(\frac{y_{\Delta,i}(s)}{\bar{y}_{\Delta,i}(s)}\big)^2\big(\bar{y}_{\Delta,i}(s)\lambda^i(\bar{y}_\Delta(s))-\frac{1}{2}|\sigma_{i}(\bar{y}_\Delta(s))|^2\big)\\
			&+|y_\Delta(s)|^{J-2}\sum_{i=1}^{d}\big(\frac{y_{\Delta,i}(s)}{\bar{y}_{\Delta,i}(s)}\big)^2|\sigma_{i}(\bar{y}_\Delta(s))|^2+\frac{J-2}{2}|y_\Delta(s)|^{J-4}\sum_{j=1}^{m}\big(\sum_{i=1}^{d}\frac{y^{2}_{\Delta,i}(s)}{\bar{y}_{\Delta,i}(s)}\sigma^{i,j}(\bar{y}_\Delta(s))\big)^2\Big)\mathrm{d}s   \\
			&-K\E\int_{0}^{t\wedge\rho_n}\Big(|y_\Delta(s)|^{-K-2}\sum_{i=1}^{d}\big(\frac{y_{\Delta,i}(s)}{\bar{y}_{\Delta,i}(s)}\big)^2\big(\bar{y}_{\Delta,i}(s)\lambda^i(\bar{y}_\Delta(s))-\frac{1}{2}|\sigma_{i}(\bar{y}_\Delta(s))|^2\big)\\
			&+|y_\Delta(s)|^{-K-2}\sum_{i=1}^{d}\big(\frac{y_{\Delta,i}(s)}{\bar{y}_{\Delta,i}(s)}\big)^2|\sigma_{i}(\bar{y}_\Delta(s))|^2-\frac{K+2}{2}|y_\Delta(s)|^{-K-4}\sum_{j=1}^{m}\big(\sum_{i=1}^{d}\frac{y^{2}_{\Delta,i}(s)}{\bar{y}_{\Delta,i}(s)}\sigma^{i,j}(\bar{y}_\Delta(s))\big)^2\Big)\mathrm{d}s   \\
			\leq&|y_0|^{J}+|y_0|^{-K}+J\E\int_{0}^{t\wedge\rho_n}|y_\Delta(s)|^{J-2}\sum_{i=1}^{d}\big(\frac{y_{\Delta,i}(s)}{\bar{y}_{\Delta,i}(s)}\big)^2\big(\bar{y}_{\Delta,i}(s)\lambda^i(\bar{y}_\Delta(s))+\frac{J-1}{2}|\sigma_{i}(\bar{y}_\Delta(s))|^2\big)\mathrm{d}s\\
			&-K\E\int_{0}^{t\wedge\rho_n}|y_\Delta(s)|^{-K-2}\sum_{i=1}^{d}\big(\frac{y_{\Delta,i}(s)}{\bar{y}_{\Delta,i}(s)}\big)^2\big(\bar{y}_{\Delta,i}(s)\lambda^i(\bar{y}_\Delta(s))-\frac{K+1}{2}|\sigma_{i}(\bar{y}_\Delta(s))|^2\big)\mathrm{d}s\\
			=:&|y_0|^{J}+|y_0|^{-K}+M_1+M_2.
		\end{flalign*}
		Using Assumption \ref{as.1}, Remark \ref{Remark1}, and Lemma \ref{LM3.1}, we derive
		\begin{flalign*}
			M_1=&J\E\int_{0}^{t\wedge\rho_n}|y_\Delta(s)|^{J-2}\sum_{i=1}^{d}\big(\frac{y_{\Delta,i}(s)}{\bar{y}_{\Delta,i}(s)}\big)^2\big(\bar{y}_{\Delta,i}(s)\lambda^i(\bar{y}_\Delta(s))+\frac{J-1}{2}|\sigma_{i}(\bar{y}_\Delta(s))|^2\big)\mathrm{d}s\\
			=&J\E\int_{0}^{t\wedge\rho_n}|y_\Delta(s)|^{J-2}\sum_{i=1}^{d}\big(\frac{y_{\Delta,i}(s)}{\bar{y}_{\Delta,i}(s)}\big)^2\big(\bar{y}_{\Delta,i}(s)\lambda^i(\bar{y}_\Delta(s))+\frac{J-1}{2}|\sigma_{i}(\bar{y}_\Delta(s))|^2\big)I_{\{\bar{y}_{\Delta,i}\geq y_i^*\}}\mathrm{d}s\\
			&+J\E\int_{0}^{t\wedge\rho_n}|y_\Delta(s)|^{J-2}\sum_{i=1}^{d}\big(\frac{y_{\Delta,i}(s)}{\bar{y}_{\Delta,i}(s)}\big)^2\big(\bar{y}_{\Delta,i}(s)\lambda^i(\bar{y}_\Delta(s))+\frac{J-1}{2}|\sigma_{i}(\bar{y}_\Delta(s))|^2\big)I_{\{\bar{y}_{\Delta,i}< y_i^*\}}\mathrm{d}s\\
			\leq & C\E\int_{0}^{t\wedge\rho_n}|y_\Delta(s)|^{J-2}\sum_{i=1}^{d}\big(\frac{y_{\Delta,i}(s)}{\bar{y}_{\Delta,i}(s)}\big)^2\big(1+\bar{y}^2_{\Delta,i}(s)\big)I_{\{\bar{y}_{\Delta,i}\geq y_i^*\}}\mathrm{d}s\\
			&+J\E\int_{0}^{t\wedge\rho_n}|y_\Delta(s)|^{J-2}\big|\frac{y_{\Delta}(s)}{\bar{y}_{\Delta}(s)}\big|^2\sum_{i=1}^{d}\big(\bar{y}_{\Delta,i}(s)|\lambda^i(\bar{y}_\Delta(s))|+\frac{J-1}{2}|\sigma_{i}(\bar{y}_\Delta(s))|^2\big)I_{\{0<\bar{y}_{\Delta,i}< y_i^*\}}\mathrm{d}s\\
			\leq &C\E\int_{0}^{t\wedge\rho_n}|y_\Delta(s)|^{J}\mathrm{d}s+C\E\int_{0}^{t\wedge\rho_n}|y_\Delta(s)|^{J-2}\big|\frac{y_{\Delta}(s)}{\bar{y}_{\Delta}(s)}\big|^2(1+|\bar{y}_\Delta(s)|^{-2\beta})I_{\bigcap_{i=1}^{d}\{ 0<\bar{y}_{\Delta,i}< y_i^*\}}\mathrm{d}s\\
			\leq&C\E\int_{0}^{t\wedge\rho_n}|y_\Delta(s)|^{J}\mathrm{d}s+C\E\int_{0}^{t\wedge\rho_n}\frac{|y_\Delta(s)|^{J-2}}{|\bar{y}_\Delta(s)|^{J-2}}\big|\frac{y_{\Delta}(s)}{\bar{y}_{\Delta}(s)}\big|^2|\bar{y}_\Delta(s)|^{J-2}(1+|\bar{y}_\Delta(s)|^{-2\beta})I_{\bigcap_{i=1}^{d}\{ 0<\bar{y}_{\Delta,i}< y_i^*\}}\mathrm{d}s\\
			\leq&C\E\int_{0}^{t\wedge\rho_n}|y_\Delta(s)|^{J}\mathrm{d}s+C\E\int_{0}^{t\wedge\rho_n}\big|\frac{y_{\Delta}(s)}{\bar{y}_{\Delta}(s)}\big|^{\bar{p}}|\bar{y}_\Delta(s)|^{J-2\beta-2}I_{\bigcap_{i=1}^{d}\{ 0<\bar{y}_{\Delta,i}< y_i^*\}}\mathrm{d}s\\
			\leq&C+C\int_{0}^{t}\E|y_\Delta(s\wedge\rho_n)|^{J}\mathrm{d}s.
		\end{flalign*}
		Besides, 
		\begin{flalign*}
			M_2\leq&-K\E\int_{0}^{t\wedge\rho_n}|y_\Delta(s)|^{-K-2}\sum_{i=1}^{d}\big(\frac{y_{\Delta,i}(s)}{\bar{y}_{\Delta,i}(s)}\big)^2\big(\bar{y}_{\Delta,i}(s)\lambda^i(\bar{y}_\Delta(s))-\frac{K+1}{2}|\sigma_{i}(\bar{y}_\Delta(s))|^2\big)\mathrm{d}s\\
			\leq&C\E\int_{0}^{t\wedge\rho_n}|y_\Delta(s)|^{-K-2}\sum_{i=1}^{d}(y_{\Delta,i}(s))^2\big(\bar{y}_{\Delta,i}(s)|\lambda^i(\bar{y}_\Delta(s))|+\frac{K+1}{2}|\sigma_{i}(\bar{y}_\Delta(s))|^2\big)I_{\bigcap_{i=1}^{d}\{ \bar{y}_{\Delta,i}\geq y_i^*\}}\mathrm{d}s\\
			\leq&C\E\int_{0}^{t\wedge\rho_n}|y_\Delta(s)|^{-K-2}\big|y_{\Delta}(s)\big|^2(1+|\bar{y}_\Delta(s)|^{2\alpha+2})I_{\bigcap_{i=1}^{d}\{ \bar{y}_{\Delta,i}\geq y_i^*\}}\mathrm{d}s\\
			\leq&C\E\int_{0}^{t\wedge\rho_n}|y_\Delta(s)|^{-K}\mathrm{d}s+C\E\int_{0}^{t\wedge\rho_n}\frac{|\bar{y}_\Delta(s)|^{K}}{|y_\Delta(s)|^{K}}|\bar{y}_\Delta(s)|^{-K+2\alpha+2}I_{\bigcap_{i=1}^{d}\{ \bar{y}_{\Delta,i}\geq y_i^*\}}\mathrm{d}s\\
			\leq&C+C\int_{0}^{t}\E|y_\Delta(s\wedge\rho_n)|^{-K}\mathrm{d}s.
		\end{flalign*}
		Therefore, we obtain
		\begin{flalign*}
			&\E[|y_\Delta(t\wedge\rho_n)|^{J}+|y_\Delta(t\wedge\rho_n)|^{-K}]\\
			=&\E[|e^{z_\Delta(t\wedge\rho_n)}|^{J}+|e^{z_\Delta(t\wedge\rho_n)}|^{-K}]\\
			\leq&|y_0|^{J}+|y_0|^{-K}+C+C\E\int_{0}^{t}|y_\Delta(s\wedge\rho_n)|^{J}\mathrm{d}s+C\int_{0}^{t}\E|y_\Delta(s\wedge\rho_n)|^{-K}\mathrm{d}s.
		\end{flalign*}
		We further derive that
		\begin{flalign*}
			\E[|y_\Delta(t\wedge\rho_n)|^{J}+|y_\Delta(t\wedge\rho_n)|^{-K}]\leq C+C\int_{0}^{t}\sup_{u\in[0,s]}\E\big[|y_\Delta(u\wedge\rho_n)|^{J}+|y_\Delta(u\wedge\rho_n)|^{-K}\big]\mathrm{d}s.
		\end{flalign*} 
		It follows that 
		\begin{flalign*}
			\sup_{ s \in[0,t]}\E[|y_\Delta(s\wedge\rho_n)|^{J}+|y_\Delta(s\wedge\rho_n)|^{-K}]\leq C\int_{0}^{t}\sup_{u\in[0,s]}\E\big[|y_\Delta(u\wedge\rho_n)|^{J}+|y_\Delta(u\wedge\rho_n)|^{-K}\big]\mathrm{d}s
		\end{flalign*} 
		for all $t\in[0,T]$. Using the Gr\"onwall inequality yields that
		\begin{flalign*}
			\sup_{ t \in[0,T]}\E[|y_\Delta(t\wedge\rho_n)|^{J}+|y_\Delta(t\wedge\rho_n)|^{-K}]\leq C.
		\end{flalign*} 
		According to the definition of $\rho_n$, we infer 
		\begin{equation}\label{Pt}
			e^{(J\wedge K)n} \PP(\rho_n\leq t)=\E[(|y_\Delta(\rho_n)|^{J}+|y_\Delta(\rho_n)|^{-K})I_{\{\rho_n\leq t\}}]\leq\E[|y_\Delta(t\wedge\rho_n)|^{J}+|y_\Delta(t\wedge\rho_n)|^{-K}]\leq C.
		\end{equation}	
		It follows that $\PP(\{\rho_{\infty}>t\})=1$, where $\rho_{\infty}:=\lim_{n\rightarrow+\infty}\rho_n$. By applying the Fatou lemma, one can get
		\begin{equation*}
			\E[|y_\Delta(t)|^{J}+|y_\Delta(t)|^{-K}]\leq \mathop{\underline{\lim}}\limits_{n\rightarrow +\infty}\E[|y_\Delta(t\wedge\rho_n)|^{J}+|y_\Delta(t\wedge\rho_n)|^{-K}]\leq C. 
		\end{equation*} 
		Finally, we obtain 
		\begin{equation}
			\sup_{t\in[0,T]}\E[|y_\Delta(t)|^{J}+|y_\Delta(t)|^{-K}]\leq C
		\end{equation}
		for $\Delta\in(0,1]$, where $C$ is dependent on $|x_0|$, $d$, $J$, and $K$.
	\end{proof}
	\begin{corollary}\label{Corollary2}
		Assuming the assumptions in Lemma \ref{pp numerical integral} holds, we obtain
		\begin{flalign}\label{corollary2}
			\PP(\rho_n \leq T)\leq \frac{C}{e^{(J\wedge K)n}}.
		\end{flalign}
	\end{corollary}
	\begin{proof}
		We derive from \eqref{Pt} that
		\begin{equation*}
			e^{(J\wedge K)n} \PP(\rho_n\leq T)\leq\E[|y_\Delta(T\wedge\rho_n)|^{J}+|y_\Delta(T\wedge\rho_n)|^{-K}]\leq C,
		\end{equation*}
		which validate \eqref{corollary2}.
	\end{proof}

	\section{The strong convergence analysis of the LTEM method}
	Regarding the LTEM method proposed in the previous section, we will derive its strong convergence rate in this section. To this end, we firstly give the following lemma.
	\begin{lemma}\label{reLM3.1}
		Suppose Assumption \ref{as.1} holds with the parameter satisfying $p\geq2$. Then for all $\Delta\in(0,1]$, there exists a constant $C$ dependent on $p$ such that 
		\begin{equation}\label{reeqLM3.1}
			\sup_{ s \in[0,T]}\E\Big[\Big|\frac{y_\Delta(s)}{\bar{y}_\Delta(s)}-1\Big|^p\Big]\leq C\Delta^{\frac{p}{2}}(\eta(\Delta))^{\frac{p}{2}}.
		\end{equation}
	\end{lemma}
	\begin{proof}
		By examining the $i$-th component of equation \eqref{anoy}, the following equation can be derived using the It\^o formula 
		\begin{flalign*}
			y_{\Delta,i}(t)=&\bar{y}_{\Delta,i}(t)+\int_{t_k}^{t}y_{\Delta,i}(s)\big(\tilde{\lambda}_\Delta^i(\bar{z}_\Delta(s))+\frac{1}{2}|\tilde{\sigma}_{\Delta,i}(\bar{z}_\Delta(s))|^2\big)\mathrm{d}s\\
			&+\int_{t_k}^{t}y_{\Delta,i}(s)\tilde{\sigma}_{\Delta,i}(\bar{z}_\Delta(s))\mathrm{d}B(s),
		\end{flalign*}
		where $y_{\Delta,i}(t)$ and $\bar{y}_{\Delta,i}(t)$ denote the $i$-th element of $y_{\Delta}(t)$ and $\bar{y}_{\Delta}(t)$, respectively. Then
		\begin{flalign}\label{yby}
			y_{\Delta}(t)=\bar{y}_{\Delta}(t)+\int_{t_k}^{t}y_{\Delta}(s)\big(\tilde{\lambda}_\Delta(\bar{z}_\Delta(s))+\frac{1}{2}|\tilde{\sigma}_{\Delta}(\bar{z}_\Delta(s))|^2\big)\mathrm{d}s+\int_{t_k}^{t}y_{\Delta}(s)\tilde{\sigma}_{\Delta}(\bar{z}_\Delta(s))\mathrm{d}B(s).
		\end{flalign}
		Applying this, \eqref{supsup} and Theorem 1.7.1 in \cite{M01} yields that 
		\begin{flalign*}
			\E\Big[\Big|\frac{y_\Delta(t)}{\bar{y}_\Delta(t)}-1\Big|^p\Big]=&\E\Big|\int_{t_k}^{t}\frac{y_\Delta(s)}{\bar{y}_\Delta(s)}\big(\tilde{\lambda}_\Delta(\bar{z}_\Delta(s))+\frac{1}{2}|\tilde{\sigma}_\Delta(\bar{z}_\Delta(s))|^2\big)\mathrm{d}s\\
			&+\int_{t_k}^{t}\frac{y_\Delta(s)}{\bar{y}_\Delta(s)}\tilde{\sigma}_\Delta(\bar{z}_\Delta(s))\mathrm{d}B(s)\Big|^p\\
			\leq& C\Delta^{p-1}\E\int_{t_k}^{t}\Big|\frac{y_\Delta(s)}{\bar{y}_\Delta(s)}\Big|^p\Big|\tilde{\lambda}_\Delta(\bar{z}_\Delta(s))+\frac{1}{2}|\tilde{\sigma}_\Delta(\bar{z}_\Delta(s))|^2\Big|^p\mathrm{d}s\\
			&+C\Delta^{\frac{p}{2}-1}\E\int_{t_k}^{t}\Big|\frac{y_\Delta(s)}{\bar{y}_\Delta(s)}\Big|^p|\tilde{\sigma}_\Delta(\bar{z}_\Delta(s))|^p\mathrm{d}s\\
			\leq&C\Delta^{\frac{p}{2}-1}(\Delta^{\frac{p}{2}}(\eta(\Delta))^p+(\eta(\Delta))^{\frac{p}{2}})\E\int_{t_k}^{t}\Big|\frac{y_\Delta(s)}{\bar{y}_\Delta(s)}\Big|^p\mathrm{d}s\\
			\leq&C\Delta^{\frac{p}{2}}(\eta(\Delta))^{\frac{p}{2}}.
		\end{flalign*}
		The assertion \eqref{reeqLM3.1} holds.
	\end{proof}
	\subsection{Strong convergence rate}
	\indent Recall the two stopping times
	\begin{flalign*}
		\rho_n^*=\inf\{t\in[0,T]:|z(t)|\geq n\}\quad\text{and}\quad \rho_n=\inf\{t\in[0,T]:|z_\Delta(t)|\geq n\}. 
	\end{flalign*}
	Set $\bar{\rho}=\rho_n^*\wedge\rho_n$ and $W_\Delta(t)=y(t)-y_\Delta(t)$.\\
	\indent To achieve the most essential results, it is necessary to impose the following assumption on $\lambda$ and $\sigma$. 
	\begin{assumption}\label{as.3}
		There exist two positive constants $p^*>2$ and $L_2$ such that the inequality
		\begin{flalign*}
			(\tilde{y}-\hat{y})^T (\lambda(\tilde{y})-\lambda(\hat{y}))+\frac{p^*-1}{2}| \sigma(\tilde{y})-\sigma(\hat{y}) |^2\leq L_2| \tilde{y}-\hat{y}|^2
		\end{flalign*}
		holds for all $\tilde{y}, \hat{y}\in \RR^d_+$.
	\end{assumption}
	\begin{lemma}\label{pmain}
		Let Assumptions \ref{as.1} and \ref{as.3} hold with the parameters satisfying $J\geq2(\alpha+1)$, $K\geq2\beta$ and $p^*>p$. Furthermore, for a given $n>|\ln y_0|$, let $\Delta\in(0,1]$ be sufficiently small such that $\psi^{-1}(\eta(\Delta))\geq n$. Then we get
		\begin{equation*}
			\sup_{ t \in[0,T]}\E[|W_\Delta(t\wedge\bar{\rho})|^p]\leq C\Delta^{\frac{p}{2}}(\eta(\Delta))^{\frac{p}{2}}.
		\end{equation*}
	\end{lemma}
	\begin{proof}
		For $s\in[0,t\wedge\bar{\rho}]$, we observe that $|\bar{z}_\Delta(s)|\leq n$. Due to the assumption $\psi^{-1}(\eta(\Delta))\geq n$, it follows that  
		\begin{flalign*}
			\tilde{\lambda}_\Delta(\bar{z}_\Delta(s))=\tilde{\lambda}(\bar{z}_\Delta(s))\quad \text{and}\quad \tilde{\sigma}_\Delta(\bar{z}_\Delta(s))=\tilde{\sigma}(\bar{z}_\Delta(s)).
		\end{flalign*}
		Focusing on the $i$-th component of equation \eqref{anoy}, the following equation can be derived by applying the It\^o formula 
		\begin{flalign*}
			e^{z_{\Delta,i}(t\wedge\bar{\rho})}=&e^{z_{i}(0)}+\int_{0}^{t\wedge\bar{\rho}}e^{z_{\Delta,i}(s)}\big(\tilde{\lambda^i}(\bar{z}_\Delta(s))+\frac{1}{2}|\tilde{\sigma}_i(\bar{z}_\Delta(s))|^2\big)\mathrm{d}s+\int_{0}^{t\wedge\bar{\rho}}e^{z_{\Delta,i}(s)}\tilde{\sigma}_i(\bar{z}_\Delta(s))\mathrm{d}B(s)\\
			=&e^{z_{i}(0)}+\int_{0}^{t\wedge\bar{\rho}}\frac{y_{\Delta,i}(s)}{\bar{y}_{\Delta,i}(s)}\lambda^i(\bar{y}_\Delta(s))\mathrm{d}s+\int_{0}^{t\wedge\bar{\rho}}\frac{y_{\Delta,i}(s)}{\bar{y}_{\Delta,i}(s)}\sigma_i(\bar{y}_\Delta(s))\mathrm{d}B(s).
		\end{flalign*}
		It follows that
		\begin{flalign*}
			y_\Delta(t\wedge\bar{\rho})=y_0+\int_{0}^{t\wedge\bar{\rho}}\frac{y_{\Delta}(s)}{\bar{y}_{\Delta}(s)}\lambda(\bar{y}_\Delta(s))\mathrm{d}s+\int_{0}^{t\wedge\bar{\rho}}\frac{y_{\Delta}(s)}{\bar{y}_{\Delta}(s)}\sigma(\bar{y}_\Delta(s))\mathrm{d}B(s).
		\end{flalign*}
		Using the It\^o formula and applying \eqref{sde} yield that
		\begin{flalign*}
			\E[|W_\Delta(t\wedge\bar{\rho})|^p]=&p\E\int_{0}^{t\wedge\bar{\rho}}|W_\Delta(s)|^{p-2}W_\Delta^T(s)\big(\lambda(y(s))-\frac{y_\Delta(s)}{\bar{y}_\Delta(s)}\lambda(\bar{y}_\Delta(s))\big)\mathrm{d}s\\
			&+\frac{p(p-1)}{2}\E\int_{0}^{t\wedge\bar{\rho}}|W_\Delta(s)|^{p-2}\Big|\sigma(y(s))-\frac{y_\Delta(s)}{\bar{y}_\Delta(s)}\sigma(\bar{y}_\Delta(s))\Big|^2\mathrm{d}s\\
			\leq&\bar{M}_1+\bar{M}_2,
		\end{flalign*}
		where 
		\begin{flalign*}
			\bar{M}_1=&p\E\int_{0}^{t\wedge\bar{\rho}}|W_\Delta(s)|^{p-2}\Big(W_\Delta^T(s)\big(\lambda(y(s))-\lambda(y_\Delta(s))\big)+\frac{p^*-1}{2}|\sigma(y(s))-\sigma(y_\Delta(s))|^2\Big)\mathrm{d}s
		\end{flalign*}
		and
		\begin{flalign*}
			\bar{M}_2=&p\E\int_{0}^{t\wedge\bar{\rho}}|W_\Delta(s)|^{p-2}W_\Delta^T(s)\big(\lambda(y_\Delta(s))-\lambda(\bar{y}_\Delta(s))+\lambda(\bar{y}_\Delta(s))-\frac{y_\Delta(s)}{\bar{y}_\Delta(s)}\lambda(\bar{y}_\Delta(s))\big)\mathrm{d}s\\
			&+\frac{p(p-1)(p^*-1)}{2(p^*-p)}\int_{0}^{t\wedge\bar{\rho}}|W_\Delta(s)|^{p-2}\Big|\sigma(y_\Delta(s))-\sigma(\bar{y}_\Delta(s))+\sigma(\bar{y}_\Delta(s))-\frac{y_\Delta(s)}{\bar{y}_\Delta(s)}\sigma(\bar{y}_\Delta(s))\Big|^2\mathrm{d}s.
		\end{flalign*}
		Here the Young inequality is used. Under Assumption \ref{as.3}, we obtain
		$\bar{M}_1\leq C\int_{0}^{t}\E|W_\Delta(s\wedge\bar{\rho})|^{p}\mathrm{d}s$ and
		\begin{flalign*}
			\bar{M}_2\leq&C\E\int_{0}^{t\wedge\bar{\rho}}|W_\Delta(s)|^p\mathrm{d}s+C\E\int_{0}^{t\wedge\bar{\rho}}\Big(|\lambda(y_\Delta(s))-\lambda(\bar{y}_\Delta(s))|^p+|1-\frac{y_\Delta(s)}{\bar{y}_\Delta(s)}|^p|\lambda(\bar{y}_\Delta(s))|^p\Big)\mathrm{d}s\\
			&+C\E\int_{0}^{t\wedge\bar{\rho}}\Big(|\sigma(y_\Delta(s))-\sigma(\bar{y}_\Delta(s))|^p+|1-\frac{y_\Delta(s)}{\bar{y}_\Delta(s)}|^p|\sigma(\bar{y}_\Delta(s))|^p\Big)\mathrm{d}s.
		\end{flalign*}
		Here the Young inequality is used. We obtain from Assumption \ref{as.1}, Remark \ref{Remark1} and the H\"older inequality that 
		\begin{flalign*}
			\bar{M}_2\leq& C\int_{0}^{t}\E|W_\Delta(s\wedge\bar{\rho})|^p\mathrm{d}s+C\int_{0}^{T}\Big(\E[1+|y_\Delta(s)|^{(1+\xi)\alpha p}+|\bar{y}_\Delta(s)|^{(1+\xi)\alpha p}+|y_\Delta(s)|^{-(1+\xi)\beta p}\\
			&+|\bar{y}_\Delta(s)|^{-(1+\xi)\beta p}]\Big)^{\frac{1}{1+\xi}}\big(\E|y_\Delta(s)-\bar{y}_\Delta(s)|^{\frac{(1+\xi)p}{\xi}}\big)^{\frac{\xi}{1+\xi}}\mathrm{d}s\\
			&+C\int_{0}^{T}\Big(\E\Big|1-\frac{y_\Delta(s)}{\bar{y}_\Delta(s)}\Big|^{\frac{(1+\xi)p}{\xi}}\Big)^{\frac{\xi}{1+\xi}}\big(\E[1+|\bar{y}_\Delta(s)|^{(1+\xi)(\alpha+1)p}+|\bar{y}_\Delta(s)|^{-(1+\xi)\beta p}]\big)^{\frac{1}{1+\xi}}\mathrm{d}s.
		\end{flalign*}
		By \eqref{yby}, Lemmas \ref{pp numerical integral} and \ref{Das1}, and Theorem 1.7.1 in \cite{M01}, we have
		\begin{flalign*}
			&\E[|y_\Delta(t)-\bar{y}_\Delta(t)|^{\frac{(1+\xi)p}{\xi}}]\\
			\leq& C\Delta^{{\frac{(1+\xi)p}{\xi}}-1}\E\int_{t_k}^{t}|y_\Delta(s)|^{\frac{(1+\xi)p}{\xi}}\Big|\tilde{\lambda}_\Delta(\bar{z}_\Delta(s))+\frac{1}{2}|\tilde{\sigma}_\Delta(\bar{z}_\Delta(s))|^2\Big|^{\frac{(1+\xi)p}{\xi}}\mathrm{d}s\\
			&+C\Delta^{\frac{(1+\xi)p}{2\xi}-1}\E\int_{t_k}^{t}|y_\Delta(s)|^{\frac{(1+\xi)p}{\xi}}|\tilde{\sigma}_\Delta(\bar{z}_\Delta(s))|^{\frac{(1+\xi)p}{\xi}}\mathrm{d}s\\
			\leq&C\Delta^{\frac{(1+\xi)p}{2\xi}}(\eta(\Delta))^{\frac{(1+\xi)p}{2\xi}}.
		\end{flalign*}
		With the aid of Lemmas \ref{pp numerical integral} and \ref{reLM3.1}, one can derive that 
		\begin{flalign}
			\bar{M}_2\leq C\int_{0}^{t}\E|W_\Delta(s\wedge\bar{\rho})|^p\mathrm{d}s+C\Delta^{\frac{p}{2}}(\eta(\Delta))^{\frac{p}{2}}.
		\end{flalign}
		Finally, the Gr\"onwall inequality implies that Lemma \ref{pmain} holds.
	\end{proof}
	
	\begin{theorem}\label{main}
		Suppose the conditions of Lemma \ref{pmain} are satisfied. If the inequality
		\begin{flalign}\label{hDi}
			\eta(\Delta)\geq \psi\Big(-\frac{J\ln(\Delta^{\frac{p}{2}}(\eta(\Delta))^{\frac{p}{2}})}{(J-p)(J\wedge K)}\Big)
		\end{flalign}
		holds for all sufficiently small $\Delta\in(0,1]$, then for all finite time $T>0$,
		\begin{equation*}
			\sup_{ t \in[0,T]}\E[|W_\Delta(t)|^p]\leq C\Delta^{\frac{p}{2}}(\eta(\Delta))^{\frac{p}{2}}.
		\end{equation*}
	\end{theorem}
	\begin{proof}
		We first perform the following decomposition
		\begin{flalign*}
			\sup_{ t \in[0,T]}\E[|W_\Delta(t)|^p]=\sup_{ t \in[0,T]}\E[|W_\Delta(t)|^pI_{\{\bar{\rho}>T\}}]+\sup_{ t \in[0,T]}\E[|W_\Delta(t)|^pI_{\{\bar{\rho}\leq T\}}]=:I_1+I_2.
		\end{flalign*}
		Using the Young inequality, Lemmas \ref{Lm2.1} and \ref{LM3.1}, Remark \ref{tau*} and Corollary  \ref{Corollary2} yields that
		\begin{flalign*}
		I_2	=&\sup_{ t \in[0,T]}\E[|W_\Delta(t)|^p\rho^{\frac{p}{J}}I_{\{\bar{\rho}\leq T\}}\rho^{-\frac{p}{J}}]\\
			\leq& \frac{p}{J}\sup_{ t \in[0,T]}\E|W_\Delta(t)|^{J}\rho+\frac{J-p}{J}\PP(\bar{\rho}\leq T)\rho^{-\frac{p}{J-p}}\\
			\leq&C\rho +C(\PP(\rho_n^*\leq T)+\PP(\rho_n\leq T))\rho^{-\frac{p}{J-p}}\\
			\leq&C\rho+Ce^{-(J\wedge K)n}\rho^{-\frac{p}{J-p}}
		\end{flalign*}
		for $\rho>0$. Choosing 
		\begin{flalign*}
			\rho = \Delta^{\frac{p}{2}}(\eta(\Delta))^{\frac{p}{2}}\quad\text{and}\quad n=-\frac{J\ln(\Delta^{\frac{p}{2}}(\eta(\Delta))^{\frac{p}{2}})}{(J-p)(J\wedge K)},
		\end{flalign*}
		we derive 
		\begin{flalign}
			I_2\leq C\Delta^{\frac{p}{2}}(\eta(\Delta))^{\frac{p}{2}}.
		\end{flalign}
		Using Lemma \ref{pmain}, we obtain
		\begin{flalign*}
			I_1\leq C\Delta^{\frac{p}{2}}(\eta(\Delta))^{\frac{p}{2}}.
		\end{flalign*}
		The proof is therefore completed.
	\end{proof}
	
	\subsection{Example and discussion}
	\begin{example}\label{Example1}
		Consider $d$-dimensional stochastic LV competition model
		\begin{flalign}\label{LV}
			\mathrm{d}y(t)&=\mathrm{diag}(y_1(t),y_2(t)\cdots,y_d(t))[f(y(t))\mathrm{d}t+\mu\mathrm{d}B(t)] \nonumber\\
			&=:\lambda(y(t))\mathrm{d}t+\sigma(y(t))\mathrm{d}B(t),
		\end{flalign}
		where $f(y)=(f^1(y),f^2(y),\cdots,f^d(t))^T=b+Ay$, the parameters $b=(b_1,b_2,\cdots,b_d)^T$, $A=(a_{ij})_{d\times d}$ and $\mu=(\mu_1,\mu_2,\cdots,\mu_d)^T$. For any $s,t\in\RR_+^d$, we define $\mathrm{L}(s,t):=\{s+l(t-s)|l\in[0,1]\}$. The mean value theorem indicates 
		\begin{flalign*}
			\lambda(s)-\lambda(t)=\lambda(w)(s-t),
		\end{flalign*}
		where a point $w\in \mathrm{L}(s,t)$. Due to $D\lambda(y)=b+2\mathrm{diag}(y_1,y_2,\cdots,y_d)A$, we can derive  
		\begin{flalign*}
			|\lambda(s)-\lambda(t)|\leq|D\lambda(u)||s-t|\leq C(1+|s|+|t|)|s-t|.
		\end{flalign*}
		Thus we see that Assumption \ref{as.1} holds with $\alpha=1$ and $\beta=0$. Under the parameter $a_{ij}\leq 0$ for all $1\leq i,j\leq d$ in \cite{YXF2024}, for any $i\in\{1,2,\cdots,d\}$ with $y_i\in(0,y_i^*)$, we derive
		\begin{flalign*}
			y_i\lambda^i(y)-\frac{K+1}{2}|\sigma_i(y)|^2=b_iy_i^2+\sum_{j=1}^{d}a_{ij}y_jy_i^2-\frac{K+1}{2}\mu_i^2y_i^2
			\geq y_i^2(b_i-\frac{K+1}{2}\mu_i^2).
		\end{flalign*}
		That is to say, if $y_i^2\Big(b_i-\frac{K+1}{2}\mu_i^2\Big)\geq 0$ for $i\in\{1,2,\cdots,d\}$, then there exists a sufficiently small $y_i^*>0$ such that for $ y_i\in(0,y_i^*)$,
		\begin{flalign*}
			y_i\lambda^i(y)-\frac{K+1}{2}|\sigma_i(y)|^2\geq0.
		\end{flalign*}
		Furthermore, for any $i\in\{1,2,\cdots,d\}$ with $y_i\in[y_i^*,\infty)$, we get
		\begin{flalign*}
			y_i\lambda^i(y)+\frac{J-1}{2}|\sigma_i(y)|^2\leq C(1+y_i^2)
		\end{flalign*}
		since the left side of the inequality above tends to negative infinite as $y_i\rightarrow \infty$. Therefore, Assumptions \ref{as.1} and \ref{as.3} hold with $y_i^2\Big(b_i-\frac{K+1}{2}\mu_i^2\Big)\geq 0$, where $i\in\{1,2,\cdots,d\}$. \\
		\indent Take $ b_1 = 2, b_2 =4,a_{11}=-4,a_{22}=-4,\mu_1=1,\mu_2= 2$ and other unspecified parameters as zero.\\
		\indent By the It\^o formula, we get 
		\begin{flalign*}
			&\mathrm{d}z_1(t)=\big(1.5-4e^{z_1(t)}\big)\mathrm{d}t+\mathrm{d}B_1(t),\nonumber\\
			&\mathrm{d}z_2(t)=\big(2-4e^{z_2(t)}\big)\mathrm{d}t+2\mathrm{d}B_2(t).
		\end{flalign*}
		Noticing that
		\begin{equation*}
			\sup_{|z|\leq v} \big(|\tilde{\lambda}(z)| \vee |\tilde{\sigma}(z)|^2 \big) \leq 4e^{r}, \quad \forall v > 0.
		\end{equation*}
		We can set $\psi(v)=4e^{r}$. Then it holds that $\psi^{-1}(r)=\ln\frac{r}{4}$. Let $\eta(\Delta)=\Delta^{-\epsilon}$ for $\epsilon\in(0,\frac{1}{2})$. Then the inequality \eqref{hDi} becomes 
		\begin{flalign*}
			\Delta^{-\epsilon}\geq 4e^{-\frac{J\ln(\Delta^{\frac{p}{2}(1-\epsilon)})}{(J-p)(J\wedge K)}}\quad i.e.\quad 1\geq 4\Delta^{\epsilon-\frac{\bar{p}p(1-\epsilon)}{(\bar{p}-p)(\bar{p}\wedge\bar{q})}}.
		\end{flalign*}
		However, for $\epsilon\in(0,\frac{1}{2})$, we can always choose sufficiently large $J$ and $K$ such that $\epsilon-\frac{Jp(1-\epsilon)}{(J-p)(J\wedge K)}>0$. Therefore, we can derive from Theorem \ref{main} that 
		\begin{equation*}
			\E[|W_\Delta(t)|^p]\leq C\Delta^{\frac{p}{2}(1-\epsilon)}, \quad\forall  \Delta\in(0,1].
		\end{equation*}
		This example exhibits that the order of $L^p$-convergence of the LTEM method is close to $\frac{p}{2}$.
	\end{example}
	
	\section{One-dimensional case}
	The LTEM method in one-dimensional case, has been proposed in \cite{PPlogTM}, opens a new chapter for explicitly solving \eqref{resde}. This work exhibits its suboptimal strong convergence rate. Subsequently, in \cite{LG2023}, the authors consider the LTEM method with weaker condition. The proposed method has optimal strong convergence rate. Regrettably, these two types of methods still have shortcomings (see e.g., Example 2 in \cite{arxivhu}). Recently, the LTEM method, further studied in \cite{TX2024}, has been applied to solve scalar SDEs with positive solutions with weak conditions. This method effectively lifts some parameter restrictions, but its strong convergence rate remains suboptimal.\\
	\indent In previous section, we derive the suboptimal convergence rate of the LTEM method for multi-dimensional SDE. The corollary of the suboptimal convergence rate in the multi-dimensional case, when restricted to the one-dimensional case, is consistent with the results of \cite{TX2024}. In this section, we will apply a proof strategy to theoretically improve the error results, thereby obtaining the optimal convergence rate of the LTEM method for the scalar SDE.\\
	\indent Consider the scalar SDE with positive solutions
	\begin{equation}\label{resde}
		\mathrm{d}y(t)=\lambda(y(t))\mathrm{d}t+\sigma(y(t))\mathrm{d}B(t),\quad 0<t\leq T, \quad y(0)=y_0\in \RR_+.
	\end{equation}
	\indent When $d=1$, Assumption \ref{as.1} degenerates into the following conditions.
	\begin{assumption}\label{reas.1}
		Suppose that $\lambda$ and $\sigma$ satisfy the non-globally Lipschitz condition: for all $\tilde{y},\hat{y}\in \RR_+$,
		\begin{flalign*}
			|\lambda(\tilde{y})-\lambda(\hat{y})|\vee|\sigma(\tilde{y})-\sigma(\hat{y})|\leq L_1(1+\tilde{y}^\alpha+\hat{y}^\alpha+\tilde{y}^{-\beta}+\hat{y}^{-\beta})|\tilde{y}-\hat{y}|,
		\end{flalign*} 
		where $L_1>0$, $\alpha>0$ and $\beta>0$. Besides, there exist some positive constants $y^*,H,K$ along with $J>1$ such that for any $\check{y}\in\RR_+$,
		\begin{flalign*}
			\left\{ 
			\begin{aligned}
				&\check{y}\lambda(\check{y})-\frac{K+1}{2}|\sigma(\check{y})|^2\geq0, &\check{y}&\in(0,y^*);\\
				&\check{y}\lambda(\check{y})+\frac{J-1}{2}|\sigma(\check{y})|^2\leq H(1+\check{y}^2), &\check{y}&\in[y^*,\infty).
			\end{aligned}
			\right.
		\end{flalign*}
	\end{assumption}
	
	\begin{remark}\label{reRemark1}
		As noted in Remark 2.1 in \cite{TX2024}, we may deduce from Assumption \ref{as.1} that there exists a constant $C_0>1$ such that 
		\begin{flalign*}
			|\lambda(y)|\leq C_0(1+y^{\alpha+1}+y^{-\beta})\quad \text{and}\quad |\sigma(y)|^2\leq C_0(1+y^{\alpha+2}+y^{-\beta+1}), \quad \forall y\in\RR_+. 
		\end{flalign*}
	\end{remark}
	
	\indent In \cite{TX2024}, the moment and inverse moment bounds of the analytical solutions, which we present in the following lemma, were derived using Assumption \ref{reas.1}.
	\begin{lemma}(Lemma 2.1 in \cite{TX2024})\label{Lm4.1}
		Suppose Assumption \ref{reas.1} holds with $\alpha\vee(\beta+1)\leq J+K$. Then SDE \eqref{resde} has unique strong solution $\{y(t)\}_{t\in[0,T]}$, and
		\begin{flalign*}
			\PP(y(t)\in\RR_+, \forall t\in[0,T])=1.
		\end{flalign*}
		Define $\theta$ is an arbitrary stopping time. It satisfies that 
		\begin{flalign*}
			\sup_{ t \in[0,T]}\E[|y(t\wedge\theta)|^{J}]<\infty  \quad \text{and} \quad \sup_{ t \in[0,T]}\E[|y(t\wedge\theta)|^{-K}]< \infty.
		\end{flalign*}
	\end{lemma}
	
	\subsection{The LTEM mthod}
	\indent To construct the LTEM method, the following transformed SDE are derived using the It\^o formula
	\begin{flalign*}
		\mathrm{d}z(t)=\tilde{\lambda}(z(t))\mathrm{d}t+\tilde{\sigma}(z(t))\mathrm{d}B(t),
	\end{flalign*}
	where 
	\begin{flalign}\label{reFGDD}
		\tilde{\lambda}(z)=e^{-z}\lambda(e^z)-\frac{1}{2}e^{-2z}|\sigma(e^z)|^2\quad\text{and}\quad \tilde{\sigma}(z)=e^{-z}\sigma(z)
	\end{flalign}
	for $z\in\RR$.\\
	\indent  We evaluate from Remark \ref{reRemark1} that 
	\begin{flalign}\label{perserve}
		|\tilde{\lambda}(z)|\vee|\tilde{\sigma}(z)|^2\leq C_0(1+e^{\alpha z}+e^{-(\beta+1)z}).
	\end{flalign}
	To begin with, we define the function $\phi(r)=4C_0e^{(\alpha\vee(\beta+1))r}$, which is strictly increasing and satisfies 
	\begin{equation*}
		\sup_{|z|\leq r} \big(|\lambda(z)| \vee |\sigma(z)|^2\big) \leq \phi(r), \quad \forall r > 0.
	\end{equation*}
	Then we defined $\phi^{-1}$ as the inverse function $\phi$, which has the property that $(4C_0, \infty) \rightarrow (0,\infty)$ and is also increasing. Besides, we choose a strictly decreasing function $\eta:(0,1]\rightarrow (4C_0, \infty)$ satisfying the following property 
	\begin{equation}\label{rehD}
		\lim_{\Delta \rightarrow 0}\eta(\Delta) =\infty \quad \mbox{and} \quad  \Delta \eta(\Delta) \leq J_0,
	\end{equation}
	where $J_0$ is a positive constant with $J_0\geq1\vee4C_0$. Fix $\Delta \in (0,1]$, let $\tilde{\lambda}_\Delta(z)$ and $\tilde{\sigma}_\Delta(z)$, referred as truncated functions, be defined as follows
	\begin{flalign*}
		\tilde{\lambda}_\Delta(z):=\left\{ 
		\begin{aligned}
			&\tilde{\lambda}\Big((|z|\wedge\phi^{-1}(\eta(\Delta)))\frac{z}{|z|}\Big), &z&\in\RR^d\setminus \{0\};\\
			&0, &z&=0
		\end{aligned}
		\right .
	\end{flalign*}
	and
	\begin{flalign*}
		\tilde{\sigma}_\Delta(z):=\left\{ 
		\begin{aligned}
			&\tilde{\sigma}\Big((|z|\wedge\phi^{-1}(\eta(\Delta)))\frac{z}{|z|}\Big), &z&\in\RR^d\setminus \{0\};\\
			&0, &z&=0. 
		\end{aligned}
		\right .
	\end{flalign*} 
	Clearly,
	\begin{equation}\label{rebbound}
		|\tilde{\lambda}_\Delta(z)|\vee|\tilde{\sigma}_\Delta(z)|^2\leq \phi(\phi^{-1}(\eta(\Delta)))=\eta(\Delta)
	\end{equation}
	for any $z\in\RR$.\\
	\indent The LTEM method for the original SDE \eqref{resde} is the special case of in Section 2, which is created by
	\begin{equation}\label{rebxx}
		y_\Delta(t)=e^{z_\Delta(t)}, \quad \forall t\in[0,T].
	\end{equation}
	
	\indent In order to develop a strong convergence theory for the LTEM method, we introduce several properties of the numerical solutions.
	\begin{lemma}\label{repp numerical integral}(Lemmas 3.1 and 3.2 in \cite{TX2024})
		For any real number $p$, it holds
		\begin{equation*}
			\sup_{\Delta\in (0,1]}\sup_{ t \in[0,T]}\E\Big[\Big|\frac{y_\Delta(t)}{\bar{y}_\Delta(t)}\Big|^p\Big]\leq C_p,
		\end{equation*}
		where $C_p$ depends on $p$. Furthermore, suppose Assumption \ref{reas.1} holds with $\alpha\vee(\beta+1)\leq J+K$ and define $\theta^*$ is an arbitrary stopping time. Then it holds
		\begin{flalign*}
			\sup_{\Delta\in (0,1]}\sup_{ t \in[0,T]}\E[|y_\Delta(t\wedge\theta^*)|^{J}]\leq C \quad \text{and} \quad \sup_{\Delta\in (0,1]}\sup_{ t \in[0,T]}\E[|y_\Delta(t\wedge\theta^*)|^{-K}]\leq C.
		\end{flalign*}
	\end{lemma}
	\subsection{The optimal convergence rate of the LTEM method in one-dimensional case}
	\indent To achieve the main results, it is necessary to impose the another condition on $\lambda$ and $\sigma$. For the case when $d=1$, Assumption \ref{as.3} is given as follows.
	\begin{assumption}\label{reas.3}
		There exist two positive constants $p^*>2$ and $L_2$ such that the inequality
		\begin{flalign*}
			(\tilde{y}-\hat{y})(\lambda(\tilde{y})-\lambda(\hat{y}))+\frac{p^*-1}{2}| \sigma(\tilde{y})-\sigma(\hat{y}) |^2\leq L_2| \tilde{y}-\hat{y}|^2
		\end{flalign*}
		holds for all $\tilde{y}, \hat{y}\in \RR_+$.
	\end{assumption}
	
	Define $\tilde{W}_\Delta(t)=y(t)-y_\Delta(t)$ and 
	\begin{flalign*}
		\theta=\inf\{t\in[0,T]:|z(t)|\geq R\}\quad\text{and}\quad \theta^*=\inf\{t\in[0,T]:|z_\Delta(t)|\geq R\}
	\end{flalign*}
	for any given $R>|\ln y_0|$, and set $\bar{\theta}=\theta\wedge\theta^{*}$. From Remark \ref{reRemark1}, we evaluate the truncated functions $\tilde{\lambda}_\Delta(x)$ and $\tilde{\sigma}_\Delta(x)$ as follows, which helps us to eliminate the infinitesimal factor $\eta(\Delta)$ in theory.
	\begin{lemma}\label{Das1}
		Suppose Assumption \ref{reas.1} holds. Then for all $\Delta\in(0,1]$,
		\begin{equation}\label{Das1eq}
			|\tilde{\lambda}_\Delta(z)|\vee
			|\tilde{\sigma}_\Delta(z)|^2\leq C(1+e^{\alpha z}+e^{-(\beta+1)z}).
		\end{equation}
	\end{lemma}
	\begin{proof}
		Fix $\Delta\in(0,1]$. For $z\in\RR$ with $|z|\leq \phi^{-1}(\eta(\Delta))$, we obtain from Assumption \ref{reas.1} and Remark \ref{reRemark1} that
		\begin{equation*}
			|\tilde{\lambda}_\Delta(z)|=|\tilde{\lambda}(z)|\leq C(1+e^{\alpha z}+e^{-(\beta+1)z}).
		\end{equation*}
		Given $z\in\RR$ with $|z|>\phi^{-1}(\eta(\Delta))$, the condition  $e^z\geq1$ allows us to derive
		\begin{flalign*}
			|\tilde{\lambda}_\Delta(z)|=&|\tilde{\lambda}(\phi^{-1}(\eta(\Delta))\frac{z}{|z|})|\leq C(1+e^{\alpha\phi^{-1}(\eta(\Delta))\frac{z}{|z|}}+e^{-(\beta+1)\psi^{-1}(\eta(\Delta))\frac{z}{|z|}}) \\
			\leq&C(2+e^{\alpha\frac{\phi^{-1}(\eta(\Delta))}{|z|}z})\leq C(1+e^{\alpha z}). 
		\end{flalign*}
		The condition $e^z<1$ allows us to derive
		\begin{flalign*}
			|\tilde{\lambda}_\Delta(z)|\leq C(2+e^{-(\beta+1)\frac{\psi^{-1}(\eta(\Delta))}{|z|}z})\leq C(1+e^{-(\beta+1)z}).
		\end{flalign*}
		Therefore, it follows
		\begin{flalign*}
			|\tilde{\lambda}_\Delta(z)|\leq C(1+e^{\alpha z}+e^{-(\beta+1)z}).
		\end{flalign*}
		Similarly, it holds
		\begin{equation*}
			|\tilde{\sigma}_\Delta(z)| \leq  C(1+e^{\alpha z}+e^{-(\beta+1)z}).
		\end{equation*}
		The assertion \eqref{Das1eq} hold.
	\end{proof}
	
	\indent The above lemma implies that the truncated functions $|\tilde{\lambda}_\Delta(z)|$ and $|\tilde{\sigma}_\Delta(z)|$ are not estimated by $\eta(\Delta)$ anymore. Furthermore, the above estimates are barely achievable in multi-dimensional scenarios, precluding us from deriving its optimal convergence rate. On the contrary, using Lemma \ref{Das1}, the moment and inverse moment bounds, we re-evaluate the following estimation without the necessary infinitesimal factors $\eta(\Delta)$, rather than the estimation $C\Delta^{\frac{p}{2}}(\eta(\Delta))^{\frac{p}{2}}$ in \cite{TX2024}. This step plays a critical role in attaining the optimal strong convergence rate.
	
	\begin{lemma}\label{LM6.1}
		Let Assumption \ref{reas.1} hold with $\frac{J}{\alpha+1}\wedge\frac{K}{\beta+1}>p$ and $p\geq2$. Then for all $\Delta\in(0,1]$, there exists a constant $C$ dependent on $p$ such that 
		\begin{equation}\label{eqLM6.1}
			\sup_{ s \in[0,T]}\E\Big[\Big|\frac{y_\Delta(s)}{\bar{y}_\Delta(s)}-1\Big|^p\Big]\leq C\Delta^{\frac{p}{2}}.
		\end{equation}
	\end{lemma}
	\begin{proof}
		Considering the one-dimensional case of Lemma \ref{reLM3.1}, one can get
		\begin{flalign}\label{ydy}
			y_{\Delta}(t)=\bar{y}_{\Delta}(t)+\int_{t_k}^{t}y_{\Delta}(s)\big(\tilde{\lambda}_\Delta(\bar{z}_\Delta(s))+\frac{1}{2}|\tilde{\sigma}_{\Delta}(\bar{z}_\Delta(s))|^2\big)\mathrm{d}s+\int_{t_k}^{t}y_{\Delta}(s)\tilde{\sigma}_{\Delta}(\bar{z}_\Delta(s))\mathrm{d}B(s).
		\end{flalign}
		By leveraging this, along with \eqref{repp numerical integral}, Lemma \ref{Das1}, H\"older's inequality, and Theorem 1.7.1 in \cite{M01}, it follows that 
		\begin{flalign*}
			&\E\Big[\Big|\frac{y_\Delta(s)}{\bar{y}_\Delta(s)}-1\Big|^p\Big]\\
			\leq&C\Delta^{p-1}\Big(\E\int_{t_k}^{s}\Big|\frac{y_\Delta(u)}{\bar{y}_\Delta(u)}\Big|^{p(1+\frac{1}{\xi})}\mathrm{d}u\Big)^{\frac{\xi}{1+\xi}}\Big(\E\int_{t_k}^{s}\Big|\tilde{\lambda}_\Delta(\bar{z}_\Delta(u))+\frac{1}{2}|\tilde{\sigma}_\Delta(\bar{z}_\Delta(u))|^2\Big|^{p(1+\xi)}\mathrm{d}u\Big)^{\frac{1}{1+\xi}}\\
			&+C\Delta^{\frac{p}{2}-1}\Big(\E\int_{t_k}^{s}\Big|\frac{y_\Delta(u)}{\bar{y}_\Delta(u)}\Big|^{p(1+\frac{1}{\xi})}\mathrm{d}u\Big)^{\frac{\xi}{1+\xi}}\Big(\E\int_{t_k}^{s}|\tilde{\sigma}_\Delta(\bar{z}_\Delta(u))|^{p(1+\xi)}\mathrm{d}u\Big)^{\frac{1}{1+\xi}}\\
			\leq&C\Delta^p(1+\E|\bar{y}_\Delta(u)|^{p\alpha(1+\xi)}+\E|\bar{y}_\Delta(u)|^{-p(\beta+1)(1+\xi)})^{\frac{1}{1+\xi}}\\
			&+C\Delta^{\frac{p}{2}}(1+\E|\bar{y}_\Delta(u)|^{\frac{p\alpha(1+\xi)}{2}}+\E|\bar{y}_\Delta(u)|^{\frac{-p(\beta+1)(1+\xi)}{2}})^{\frac{1}{1+\xi}}.
		\end{flalign*}
		Under the condition $\frac{J}{\alpha+1}\wedge\frac{K}{\beta+1}>p$, there exists $\xi>0$ such that  $\frac{J}{\alpha+1}\wedge\frac{K}{\beta+1}\geq(1+\xi)p$. It means that 
		\begin{flalign}\label{as.4}
			J\geq p(\alpha+1)(1+\xi) \quad\text{and}\quad K\geq p(\beta+1)(1+\xi).
		\end{flalign}
		Since $p\geq 2$, we have $J+K>\alpha\wedge\beta+1$, we can derive from Lemma \ref{repp numerical integral} that the assertion \eqref{eqLM6.1} holds.
	\end{proof}
	
	\begin{lemma}\label{repmain}
		Suppose Assumptions \ref{reas.1} and \ref{reas.3} hold with $\frac{J}{\alpha+1}\wedge\frac{K}{\beta+1}>p$. Given $R>|\ln y_0|$, let $\theta$ and $\theta^*$ be the stopping times defined above. Let $\Delta\in(0,1]$ be sufficiently small such that $\phi^{-1}(\eta(\Delta))\geq R$. Then we obtain
		\begin{equation*}
			\sup_{ t \in[0,T]}\E[|\tilde{W}_\Delta(t\wedge\bar{\theta})|^p]\leq C\Delta^{\frac{p}{2}}.
		\end{equation*}
	\end{lemma}
	\begin{proof}
		For $s\in[0,t\wedge\bar{\theta}]$, we observe that $|\bar{y}_\Delta(s)|\leq R$. Due to the assumption $\phi^{-1}(\eta(\Delta))\geq R$, it follows that  $\tilde{\lambda}_\Delta(\bar{y}_\Delta(s))=\tilde{\lambda}(\bar{y}_\Delta(s))$ and $\tilde{\sigma}_\Delta(\bar{y}_\Delta(s))=\tilde{\sigma}(\bar{y}_\Delta(s))$ for $s\in[0,t\wedge\bar{\theta}]$. 
		\iffalse
		By applying the It\^o formula and using \eqref{reFGDD}, we have
		\begin{flalign*}
			e^{z_\Delta(t)}=&e^{z_0}+\int_{0}^{t}e^{z_\Delta(s)}\big(\tilde{\lambda}_\Delta(\bar{z}_\Delta(s))+\frac{1}{2}|\tilde{\sigma}_\Delta(\bar{z}_\Delta(s))|^2\big)\mathrm{d}s+\int_{0}^{t}e^{z_\Delta(s)}\tilde{\sigma}_\Delta(\bar{z}_\Delta(s))\mathrm{d}B(s)\\
			=&y_0+\int_{0}^{t}\frac{y_\Delta(s)}{\bar{y}_\Delta(s)}\lambda(\bar{y}_\Delta(s))\mathrm{d}s+\int_{0}^{t}\frac{y_\Delta(s)}{\bar{y}_\Delta(s)}\sigma(\bar{y}_\Delta(s))\mathrm{d}B(s).
		\end{flalign*}
		Therefore, we have
		\begin{flalign*}
			y(t)-y_\Delta(t)=\int_{0}^{t}\big(\lambda(y(s))-\frac{y_\Delta(s)}{\bar{y}_\Delta(s)}\lambda(\bar{y}_\Delta(s))\big)\mathrm{d}s+\int_{0}^{t}\big(\sigma(y(s))-\frac{y_\Delta(s)}{\bar{y}_\Delta(s)}\sigma(\bar{y}_\Delta(s))\big)\mathrm{d}B(s).
		\end{flalign*}
		Using the It\^o formula, we have
		\fi
		It is similar to the one-dimensional case of Lemma \ref{pmain}, thus we obtain
		\begin{flalign*}
			\E[|\tilde{W}_\Delta(t\wedge\bar{\theta})|^p]=&p\E\int_{0}^{t\wedge\bar{\theta}}|\tilde{W}_\Delta(s)|^{p-2}\tilde{W}_\Delta(s)\big(\lambda(y(s))-\frac{y_\Delta(s)}{\bar{y}_\Delta(s)}\lambda(\bar{y}_\Delta(s))\big)\mathrm{d}s\\
			&+\frac{p(p-1)}{2}\E\int_{0}^{t\wedge\bar{\theta}}|\tilde{W}_\Delta(s)|^{p-2}\Big|\sigma(y(s))-\frac{y_\Delta(s)}{\bar{y}_\Delta(s)}\sigma(\bar{y}_\Delta(s))\Big|^2\mathrm{d}s\\
			\leq&\hat{I_1}+\hat{I_2},
		\end{flalign*}
		where 
		\begin{flalign*}
			\hat{I_1}=&p\E\int_{0}^{t\wedge\bar{\theta}}|\tilde{W}_\Delta(s)|^{p-2}\Big(\tilde{W}_\Delta(s)\big(\lambda(y(s))-\lambda(y_\Delta(s))\big)+\frac{p^*-1}{2}|\sigma(y(s))-\sigma(y_\Delta(s))|^2\Big)\mathrm{d}s
		\end{flalign*}
		and
		\begin{flalign*}
			\hat{I_2}=&p\int_{0}^{t\wedge\bar{\theta}}|\tilde{W}_\Delta(s)|^{p-2}\tilde{W}_\Delta(s)\big(\lambda(y_\Delta(s))-\frac{y_\Delta(s)}{\bar{y}_\Delta(s)}\lambda(\bar{y}_\Delta(s))\big)\mathrm{d}s\\
			&+\frac{p(p-1)(p^*-1)}{2(p^*-p)}\int_{0}^{t\wedge\bar{\theta}}|\tilde{W}_\Delta(s)|^{p-2}\Big|\sigma(y_\Delta(s))-\frac{y_\Delta(s)}{\bar{y}_\Delta(s)}\sigma(\bar{y}_\Delta(s))\Big|^2\mathrm{d}s.
		\end{flalign*}
		Here the Young inequality is used. Under Assumption \ref{reas.3}, we obtain
		$\hat{I_1}\leq C\int_{0}^{t}\E|\tilde{W}_\Delta(s\wedge\bar{\theta})|^{p}\mathrm{d}s$, and derive from the Young inequality that
		\begin{flalign*}
			\hat{I_2}\leq&C\E\int_{0}^{t\wedge\bar{\theta}}|\tilde{W}_\Delta(s)|^{p-1}\Big|\lambda(y_\Delta(s))-\lambda(\bar{y}_\Delta(s))+\lambda(\bar{y}_\Delta(s))-\frac{y_\Delta(s)}{\bar{y}_\Delta(s)}\lambda(\bar{x}_\Delta(s))\Big|\mathrm{d}s\\
			&+C\E\int_{0}^{t\wedge\bar{\theta}}|\tilde{W}_\Delta(s)|^{p-2}\Big|\sigma(y_\Delta(s))-\sigma(\bar{y}_\Delta(s))+\sigma(\bar{y}_\Delta(s))-\frac{y_\Delta(s)}{\bar{y}_\Delta(s)}\sigma(\bar{y}_\Delta(s))\Big|^2\mathrm{d}s\\
			\leq&C\E\int_{0}^{t\wedge\bar{\theta}}|\tilde{W}_\Delta(s)|^p\mathrm{d}s+C\E\int_{0}^{t\wedge\bar{\theta}}\Big(|\lambda(y_\Delta(s))-\lambda(\bar{y}_\Delta(s))|^p+|1-\frac{y_\Delta(s)}{\bar{y}_\Delta(s)}|^p|\lambda(\bar{y}_\Delta(s))|^p\Big)\mathrm{d}s\\
			&+C\E\int_{0}^{t\wedge\bar{\theta}}\Big(|\sigma(y_\Delta(s))-\sigma(\bar{y}_\Delta(s))|^p+|1-\frac{y_\Delta(s)}{\bar{y}_\Delta(s)}|^p|\sigma(\bar{y}_\Delta(s))|^p\Big)\mathrm{d}s.
		\end{flalign*}
		Using Assumption \ref{reas.1}, Remark \ref{reRemark1} and the H\"older inequality, we obtain 
		\begin{flalign*}
			\hat{I_2}\leq& C\int_{0}^{t}\E|\tilde{W}_\Delta(s\wedge\bar{\theta})|^p\mathrm{d}s+C\int_{0}^{T}\Big(\E[1+|y_\Delta(s)|^{\alpha(1+\xi)p}+|\bar{y}_\Delta(s)|^{\alpha(1+\xi)p}+|y_\Delta(s)|^{-\beta(1+\xi)p}\\
			&+|\bar{y}_\Delta(s)|^{-\beta(1+\xi)p}]\Big)^{\frac{1}{1+\xi}}\big(\E|y_\Delta(s)-\bar{y}_\Delta(s)|^{\frac{(1+\xi)p}{\xi}}\big)^{\frac{\xi}{1+\xi}}\mathrm{d}s\\
			&+C\int_{0}^{T}\Big(\E\Big|1-\frac{y_\Delta(s)}{\bar{y}_\Delta(s)}\Big|^{\frac{(1+\xi)p}{\xi}}\Big)^{\frac{\xi}{1+\xi}}\big(\E[1+|\bar{y}_\Delta(s)|^{(\alpha+1)(1+\xi)p}+|\bar{y}_\Delta(s)|^{-\beta(1+\xi)p}]\big)^{\frac{1}{1+\xi}}\mathrm{d}s\\
			&+C\int_{0}^{T}\Big(\E\Big|1-\frac{y_\Delta(s)}{\bar{y}_\Delta(s)}\Big|^{\frac{(1+\xi)p}{\xi}}\Big)^{\frac{\xi}{1+\xi}}\big(\E[1+|\bar{y}_\Delta(s)|^{\frac{(\alpha+2)(1+\xi)p}{2}}+|\bar{y}_\Delta(s)|^{\frac{-(\beta-1)(1+\xi)p}{2}}]\big)^{\frac{1}{1+\xi}}\mathrm{d}s.
		\end{flalign*}
		By \eqref{ydy}, \eqref{as.4}, Lemmas \ref{repp numerical integral} and \ref{Das1}, the H\"older inequality and Theorem 1.7.1 in \cite{M01}, we have
		\begin{flalign*}
			&\E[|y_\Delta(s)-\bar{y}_\Delta(s)|^{\frac{(1+\xi)p}{\xi}}]\\
			\leq& C\Delta^{{\frac{(1+\xi)p}{\xi}}-1}\E\int_{t_k}^{s}|y_\Delta(u)|^{\frac{(1+\xi)p}{\xi}}\Big|\tilde{\lambda}(\bar{z}_\Delta(u))+\frac{1}{2}|\tilde{\sigma}(\bar{z}_\Delta(u))|^2\Big|^{\frac{(1+\xi)p}{\xi}}\mathrm{d}u\\
			&+C\Delta^{\frac{(1+\xi)p}{2\xi}-1}\E\int_{t_k}^{s}|y_\Delta(u)|^{\frac{(1+\xi)p}{\xi}}|\tilde{\sigma}(\bar{y}_\Delta(u))|^{\frac{(1+\xi)p}{\xi}}\mathrm{d}u\\
			\leq&C\Delta^{{\frac{(1+\xi)p}{\xi}}-1}\Big(\E\int_{t_k}^{s}|y_\Delta(u)|^{\frac{(\xi+1)p}{\xi-1}}\mathrm{d}u\Big)^{\frac{\xi-1}{\xi}}\Big(\E\int_{t_k}^{s}\Big|\tilde{\lambda}(\bar{z}_\Delta(u))+\frac{1}{2}|\tilde{\sigma}(\bar{z}_\Delta(u))|^2\Big|^{(1+\xi)p}\mathrm{d}u\Big)^{\frac{1}{\xi}}\\
			&+C\Delta^{\frac{(1+\xi)p}{2\xi}-1}\Big(\E\int_{t_k}^{s}|y_\Delta(u)|^{\frac{(\xi+1)p}{\xi-1}}\mathrm{d}u\Big)^{\frac{\xi-1}{\xi}}\Big(\E\int_{t_k}^{s}|\tilde{\sigma}(\bar{z}_\Delta(u))|^{(1+\xi)p}\mathrm{d}u\Big)^{\frac{1}{\xi}}\\
			\leq&C\Delta^{{\frac{(1+\xi)p}{\xi}}}(1+\E|\bar{y}_\Delta(u)|^{p\alpha(1+\xi)}+\E|\bar{y}_\Delta(u)|^{-p(\beta+1)(1+\xi)})^{\frac{1}{\xi}}\\
			&+C\Delta^{\frac{(1+\xi)p}{2\xi}}(1+\E|\bar{y}_\Delta(u)|^{\frac{p\alpha(1+\xi)}{2}}+\E|\bar{y}_\Delta(u)|^{\frac{-p(\beta+1)(1+\xi)}{2}})^{\frac{1}{\xi}}\\
			\leq&C\Delta^{\frac{(1+\xi)p}{2\xi}}.
		\end{flalign*}
		With the aid of \eqref{as.4}, Lemmas \ref{repp numerical integral} and \ref{LM6.1} can be used to yield that
		\begin{flalign}
			\hat{I_2}\leq C\int_{0}^{t}\E|\tilde{W}_\Delta(s\wedge\bar{\theta})|^p\mathrm{d}s+C\Delta^{\frac{p}{2}}.
		\end{flalign}
		Finally, the Gr\"onwall inequality implies that Lemma \ref{repmain} holds.
	\end{proof}
	
	\begin{theorem}\label{remain}
	Suppose the conditions of Lemma \ref{repmain} are satisfied. If 
		\begin{flalign}\label{rehDi}
			\eta(\Delta)\geq \phi\Big(-\frac{Jp\ln\Delta}{2(J-p)(J\wedge K)}\Big)
		\end{flalign}
		holds for all sufficiently small $\Delta\in(0,1]$, then we have
		\begin{equation*}
			\sup_{ t \in[0,T]}\E[|\tilde{W}_\Delta(t)|^p]\leq C\Delta^{\frac{p}{2}}
		\end{equation*}
		for any fixed $T=N\Delta>0$.
	\end{theorem}
	\begin{proof}
		We first perform the following decomposition
		\begin{flalign*}
		\sup_{ t \in[0,T]}\E[|\tilde{W}_\Delta(t)|^p]=\sup_{ t \in[0,T]}\E[|\tilde{W}_\Delta(t)|^pI_{\{\bar{\theta}>T\}}]+\sup_{ t \in[0,T]}\E[|\tilde{W}_\Delta(t)|^pI_{\{\bar{\theta\leq T}\}}]=:\bar{I}_1+\bar{I}_2.
		\end{flalign*}
		Using the Young inequality, Lemmas \ref{Lm4.1} and \ref{repp numerical integral} yields that
		\begin{flalign*}
			\bar{I}_2=&\sup_{ t \in[0,T]}\E[|\tilde{W}_\Delta(t)|^p\delta^{\frac{p}{J}}I_{\{\bar{\theta}\leq T\}}\delta^{-\frac{p}{J}}]\\
			\leq& \frac{p}{J}\sup_{ t \in[0,T]}\E|\tilde{W}_\Delta(t)|^{J}\delta+\frac{J-p}{J}\PP(\bar{\theta}\leq T)\delta^{-\frac{p}{J-p}}\\
			\leq&C\delta +C(\PP(\theta\leq T)+\PP(\theta^*\leq T))\delta^{-\frac{p}{J-p}}\\
			\leq&C\delta+C\Big(\frac{\E[|y(T\wedge\bar{\theta})|^{J}]+\E[|y(T\wedge\bar{\theta})|^{-K}]}{e^{(J\wedge K)R}}+\frac{\E[|y_\Delta(T\wedge\bar{\theta})|^{J}]+\E[|y_\Delta(T\wedge\bar{\theta})|^{-K}]}{e^{(J\wedge K)R}}\Big)\delta^{-\frac{p}{J-p}}\\
			\leq&C\delta+Ce^{-(J\wedge K)R}\delta^{-\frac{p}{J-p}}.
		\end{flalign*}
		Choosing 
		\begin{flalign*}
			\delta = \Delta^{\frac{p}{2}}\quad\text{and}\quad R=-\frac{Jp\ln\Delta}{2(J-p)(J\wedge K)},
		\end{flalign*}
		we have 
		\begin{flalign}
			\bar{I}_2\leq C\Delta^{\frac{p}{2}}.
		\end{flalign}
		Using Lemma \ref{repmain}, we obtain
		\begin{flalign*}
			\bar{I}_1\leq C\Delta^{\frac{p}{2}}.
		\end{flalign*}
		The proof is therefore completed. 
	\end{proof}

	\section{Numerical examples}
	In this section, we will explore one example and present simulations to demonstrate the advantages and efficiency of our new results. Before discussing the numerical examples, it is necessary to present the following specifications. The expression for evaluating the strong convergence error in the $L^1(\Omega)$-norm is as follows: 
	\begin{equation*}
		\E[|y(T)-y_T| ]= \frac{1}{M}\sum_{i=1}^{M}\Big|y^i(T) - y^i_T\Big|,
	\end{equation*}
	where $T$ is the terminal time and $M$ represent the number of trajectories, while $y^i(T)$ and $y^i_T$ present the $i$-th exact solution and numerical solution, respectively. Throughout our numerical experiments, the reference solution is generated via the LTEM method with a step size of $\Delta = 2^{-17}$. To investigate convergence rates, we compute numerical solutions with different step sizes of $\Delta = 2^{-14}, 2^{-13}, 2^{-12}, 2^{-11}$ and $2^{-10}$.\\
	\indent We proceed the continuity of Example \ref{Example1}. Take the initial value $y_0=(y_1(0), y_2(0))^T = (1,2)^T$ and other parameters as Example \ref{Example1}. As shown in Figure \ref{figure1}, the strong convergence order of the LTEM method is close to 1, which beyonds theoretical result in Theorem \ref{main}. Actually, by using the logarithmic transformation, the noise of the transformed SDE becomes additive. Therefore, the result of the first order is predictable.\\
	\indent Besides, we take the parameters as follows: $y_1(0) = 1, y_2(0) = 2,b_1 = 10, b_2 =6,a_{11}=-10,a_{22}=-8,\sigma_1=3$ and $\sigma_2= 2 $, with all other unspecified parameters set to zero. We generate 10 trajectories of the numerical solutions using both the truncated EM and LTEM methods with the step size $\Delta=2^{-5}$ over the time interval $[0, 2]$. In Figure \ref{figure2}, the numerical solutions generated by the truncated EM method exhibit negative values. In contrast, the LTEM method ensures that the values remain positive at all times. 
	\begin{figure}[htbp]
		\centering
		\includegraphics[width=11cm,height=8cm]{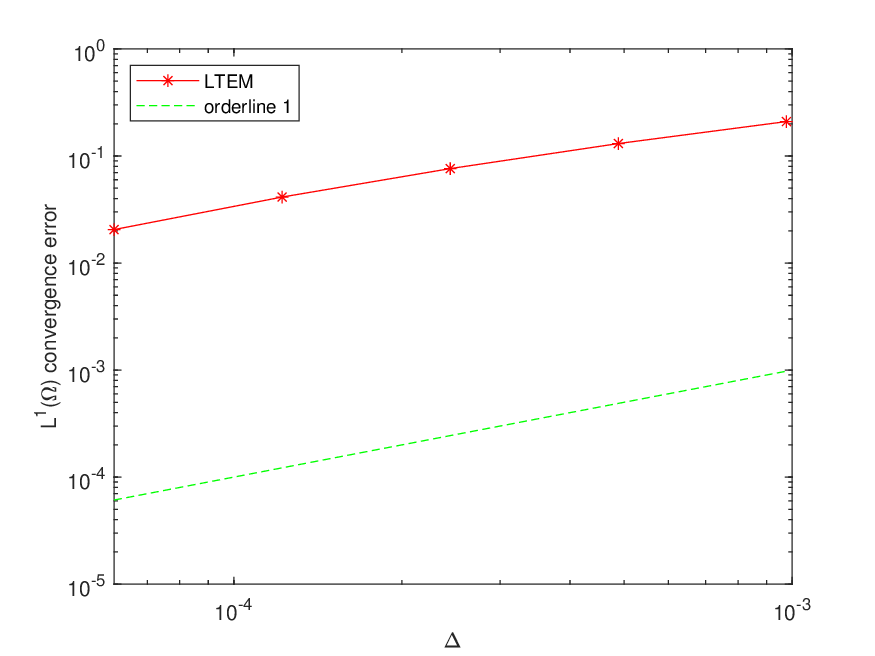}		
		\caption{Convergence rate }
		\label{figure1}
	\end{figure}
	
	\begin{figure}[htbp]
		\centering
		\begin{minipage}{0.4\linewidth}
			\centering
			\includegraphics[height = 5cm, width=5cm]{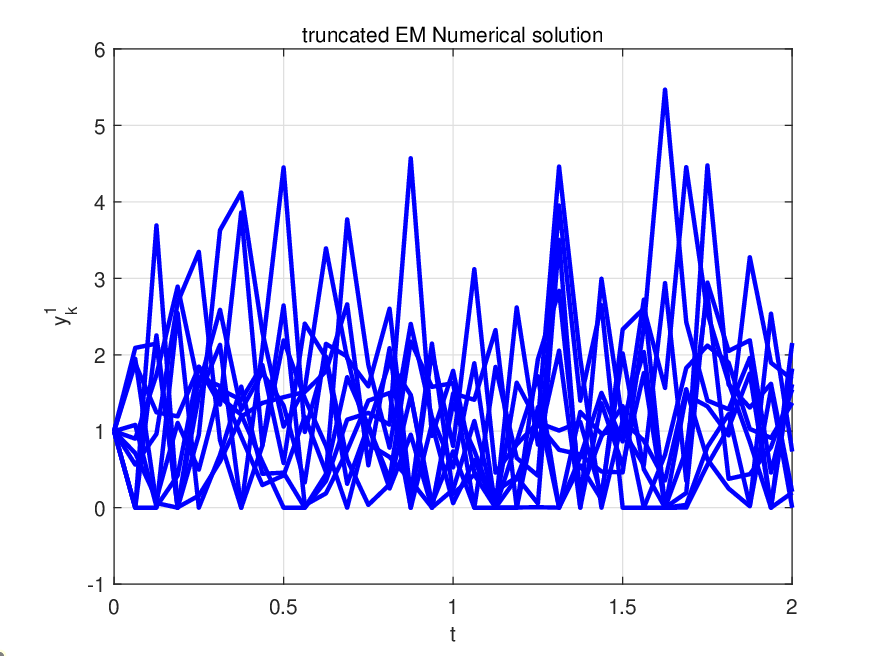}
		\end{minipage}
		\begin{minipage}{0.4\linewidth}
			\centering
			\includegraphics[height = 5cm, width=5cm]{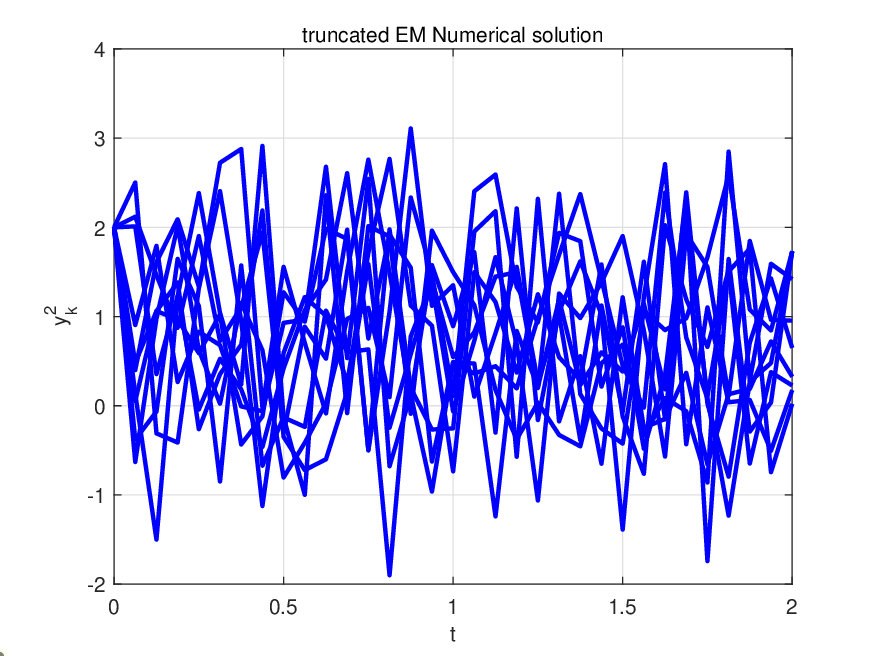}
		\end{minipage}
		\begin{minipage}{0.4\linewidth}
			\centering
			\includegraphics[height = 5cm, width=5cm]{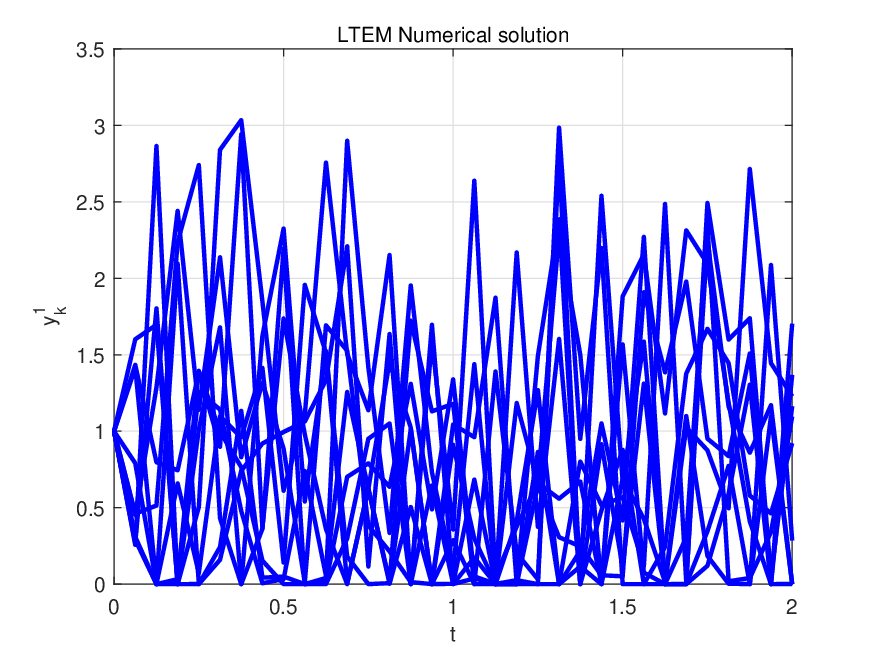}
		\end{minipage}
		\begin{minipage}{0.4\linewidth}
			\centering
			\includegraphics[height = 5cm, width=5cm]{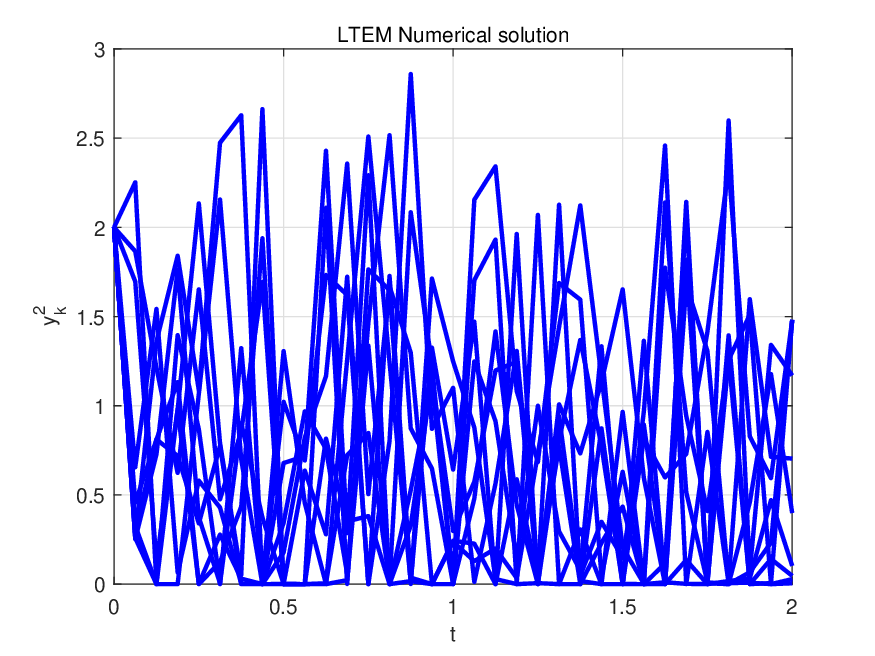}
		\end{minipage}
		\caption{10 trajectories of numerical solutions $y_k^1$ and $y_k^2$ generated by the truncated EM and LTEM methods for the stochastic LV model, using a step size $\Delta=2^{-6}$ and $T=2$.}
		\label{figure2}
	\end{figure}
	
	\begin{example}\label{Example2}
		Consider the $3$-dimensional Lotka--Volterra system 
		\begin{flalign}\label{LV3eq}
			&\mathrm{d}y_1(t)=\big(50y_1(t)-55y^2_1(t)\big)\mathrm{d}t+y_1(t)\Big(7+\frac{\sin(y_1(t))+\sin(y_2(t))+\sin(y_3(t))}{1+y_1(t)+y_2(t)+y_3(t)}\Big)\mathrm{d}B(t),\nonumber\\
			&\mathrm{d}y_2(t)=\big(30y_2(t)-10y^2_2(t)\big)\mathrm{d}t+y_2(t)\Big(2+\frac{y_1(t)+y_2(t)+y_3(t)}{1+(y_1(t)+y_2(t)+y_3(t))^2}\Big)\mathrm{d}B(t),\\
			&\mathrm{d}y_3(t)=\big(20y_3(t)-15y^2_3(t)\big)\mathrm{d}t+y_3(t)\Big(5+\frac{\cos(y_1(t))+\cos(y_2(t))}{1+y_3^2(t)}\Big)\mathrm{d}B(t)\nonumber
		\end{flalign}
		with $y_1(0) = 0.5, y_2(0) = 2,y_3(0) = 1$. From Example \ref{Example1}, we can verify that Assumptions \ref{as.1} and \ref{as.3} hold with $\alpha=1$ and $\beta=0$.\\
		\indent By the It\^o formula, we get 
		\begin{flalign*}
			&\mathrm{d}z_1(t)=\big(50-50e^{z_1(t)}-0.5N_1^2\big)\mathrm{d}t+N_1\mathrm{d}B(t),\nonumber\\
			&\mathrm{d}z_2(t)=\big(30-10e^{z_2(t)}-0.5N_2^2\big)\mathrm{d}t+N_2\mathrm{d}B(t),\\
			&\mathrm{d}z_3(t)=\big(20-15e^{z_3(t)}-0.5N_3^2\big)\mathrm{d}t+N_3\mathrm{d}B(t),\nonumber
		\end{flalign*}
		where $N_1=\Big(7+\frac{\sin(e^{z_1(t)}+\sin(e^{z_2(t)}+\sin(e^{z_3(t)}}{1+e^{z_13(t)}+e^{z_2(t)}+e^{z_3(t)}}\Big)$, $N_2=\Big(2+\frac{e^{z_1(t)}+e^{z_2(t)}+e^{z_3(t)}}{1+(e^{z_1(t)}+e^{z_2(t)}+e^{z_3(t)})^2}\Big)$ and $N_3=\Big(5+\frac{\cos(e^{z_1(t)})+\cos(e^{z_2(t)})}{1+e^{z_3(t)}}\Big)$. We define the function $\psi(r)=50e^{r}$, for which the corresponding inverse function is given by $\psi^{-1}(r)=\ln\frac{r}{50}$. We also define $\eta(\Delta)=50\Delta^{-0.5}$, which satisfies the condition stated in \eqref{hD}. In Figure \ref{figure3}, we observe that the strong convergence order of the LTEM method is close to 1/2, which consists with theoretical result in Theorem \ref{main}. Furthermore, we observe that our method consistently preserves positivity as shown in Table \ref{tab}. \\
		\indent By combining Figure \ref{figure2} and Table \ref{tab}, we can see that the LTEM method is better at preserving positivity than the truncated EM method.
		\begin{figure}[htbp]
			\centering
			\includegraphics[width=11cm,height=8cm]{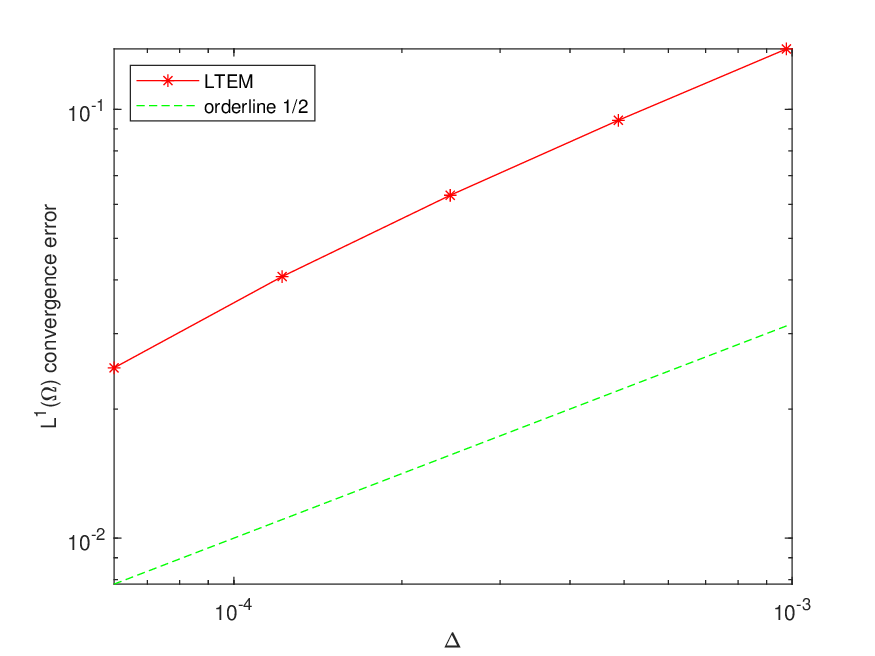}		
			\caption{Convergence rate }
			\label{figure3}
		\end{figure}
	\end{example}
	
	\begin{table}[htbp]% [htbp] -- 表格放置位置
		\centering
		\renewcommand{\arraystretch}{1.2}
		\caption{The percentages of non-positive numerical values of $y_k^1$, $y_k^2$ and $y_k^3$ produced by the truncated EM and LTEM methods, using different \(T\) and \(\Delta\), are based on \(10^5\) sample paths for Example \ref{Example2}.}  % 标题
		\begin{center}  % 表格整体居中	
			\begin{tabular}{c c c c c}
				\toprule[1.2pt]
				Solution  &Time \qquad & $\Delta$\hspace{1em} \qquad &Truncated EM  \hspace{1em}\qquad & LTEM \qquad \\
				\hline
				$y_k^1$\hspace{2em} & \tabincell{c}{ $T=2$ \\ $T=4$ \\ $T=8$ } & \tabincell{c}{$2^{-11}$\\$2^{-10}$\\$2^{-9}$} & \tabincell{c}{0\\0.83\\62.31} & \tabincell{c}{0\\0\\0} \\
				\hline
				$y_k^2$\hspace{2em} & \tabincell{c}{  $T=2$ \\ $T=4$ \\ $T=8$ } & \tabincell{c}{$2^{-11}$\\$2^{-10}$\\$2^{-9}$} & \tabincell{c}{0\\0\\0.64} & \tabincell{c}{0\\0\\0}\\
				\hline
				$y_k^3$\hspace{2em} & \tabincell{c}{ $T=2$ \\ $T=4$ \\ $T=8$ } & \tabincell{c}{$2^{-11}$\\$2^{-10}$\\$2^{-9}$} & \tabincell{c}{0\\0\\1.05} & \tabincell{c}{0\\0\\0}\\
				\bottomrule[1.2pt]   % 底线，也可以使用[1.5pt]
			\end{tabular}
			\label{tab}
		\end{center}
	\end{table}
	
	\section{Conclusion}
	\indent In this paper, we focus on the LTEM method for the general SDEs with positive solutions. The primary result of this paper is that we have successfully extended the LTEM method to the multi-dimensional setting and derived its suboptimal convergence rate; in other words, the proposed method is now capable of solving general multi-dimensional SDEs with positive solutions and its convergence rate close to 1/2. Secondly, by eliminating unnecessary infinitesimal factors $\eta(\Delta)$, we achieve a theoretical enhancement in the strong convergence rate of the LTEM method in one-dimensional case, elevating it from suboptimal to optimal. Finally, the numerical results align with our theoretical conclusions, confirming both the positivity-preserving property and the strong convergence rate of the LTEM method.\\
	
	\noindent\small\textbf{Author Contribution}~~Xingwei Hu: Data curation, Formal analysis, Investigation, Methodology, Software, Visualization, Writing -- original draft, Writing -- review \& editing. Xinjie Dai: Conceptualization, Formal analysis, Funding acquisition, Supervision, Validation, Writing -- review \& editing. Aiguo Xiao: Conceptualization, Formal analysis, Funding acquisition, Methodology, Project administration, Resources, Supervision, Validation, Writing -- original draft, Writing -- review \& editing.\\
	\\
	%\noindent\small\textbf{Funding}~~This work was supported by National Natural Science Foundation of China (Nos.12471391 and 12401547).\\
	%\\
	\noindent\small\textbf{Data Availability}~~Data sharing is not applicable to this article as no datasets were generated or analyzed during the current study.
	
	\section*{Declarations}
	\noindent\small\textbf{Competing Interests }~~The authors declare no competing interests.



\begin{thebibliography}{99}
		%\bibitem{pj2016}
		%Beyn, W.-J., Isaak, E., Kruse, R.: Stochastic C-stability and B-consistency of explicit and implicit Euler-type schemes. J. Sci. Comput. \textbf{67}, 955--987 (2016) 
		
		%\bibitem{pj2017}
		%Beyn, W.-J., Isaak, E., Kruse, R.: Stochastic C-stability and B-consistency of explicit and implicit Milstein-type schemes. J. Sci. Comput. \textbf{70}, 1042--1077 (2017)
		
		\bibitem{YQX2024}
		Y. Cai, Q. Guo, X. Mao, Strong convergence of an explicit numerical approximation for $n$-dimensional superlinear SDEs with positive solutions, Math. Comput. Simul. 216 (2024) 198--212. 
		
		\bibitem{YXF2024}
		Y. Cai, X. Mao, F. Wei, An advanced numerical scheme for multi-dimensional stochastic Kolmogorov equations with superlinear coefficients, J. Comput. Appl. Math. 437 (2024) 115472. 
		
		\bibitem{CJI2016}
		J.-F. Chassagneux, A. Jacquier, I. Mihaylov, An explicit Euler scheme with strong rate of convergence for financial SDEs with non-Lipschitz coefficients, SIAM J. Financ. Math. 7 (1) (2016) 993--1021.
		
		%\bibitem{GLMY2017}
		%Guo, Q., Liu, W., Mao, X., Yue, R.: The truncated Milstein method for stochastic differential equations with commutative noise. J. Comput. Appl. Math. \textbf{338}, 298--310 (2018)
		
		%\bibitem{M18}
		%Hu, L., Li, X., Mao, X.: Convergence rate and stability of the truncated Euler--Maruyama method for stochastic differential equations. J. Comput. Appl. Math. \textbf{337}, 274--289 (2018)
		
		%\bibitem{Hut02}
		%Hutzenthaler, M., Jentzen, A., Kloeden, P.E.: Strong convergence of an explicit numerical method for SDEs with non-globally Lipschitz continuous coefficients. Ann. Appl. Probab. \textbf{22} (4), 1611--1641 (2012) 
		
		\bibitem{arxivhu}
		X. Hu, X. Dai, A. Xiao, A positivity-preserving truncated Euler--Maruyama method for stochastic differential equations with positive solutions: multi-dimensional case, (2024). \textcolor{blue}{arXiv:2412.20988 }
		
		\bibitem{HXW}
		X. Hu, M. Wang, X. Dai, Y. Yu, A. Xiao, A positivity preserving Milstein-type method for stochastic differential equations with positive solutions, J. Comput. Appl. Math. 449 (2024) 115963. 
		
		%\bibitem{LM13}
		%Liu, W., Mao, X.: Strong convergence of the stopped Euler–Maruyama method for nonlinear stochastic differential equations. Appl. Math. Comput. \textbf{223}, 389--400 (2013) 
		
		%\bibitem{Li2019}
		%Li, X., Mao, X., Yin, G.: Explicit numerical approximations for stochastic differential equations in finite and infinite horizons: Truncation methods,
		%convergence in pth moment and stability. IMA J. Numer. Anal. \textbf{39}, 847--892 (2019) 
		
		\bibitem{LNX2025}
		R. Liu, A. Neuenkirch, X. Wang, A strong order 1.5 boundary preserving discretization scheme for scalar SDEs defined in a domain, Math. Comput. 94 (354) (2025)  1815--1862. 
		
		\bibitem{APPL2023}
		Y. Li, W. Cao, A positivity preserving Lamperti transformed Euler--Maruyama method for solving the stochastic Lotka--Volterra competition model, Commun. Nonlinear Sci. Numer. Simul. 122 (2023) 107260. 
		
		\bibitem{LG2023}
		Z. Lei, S. Gan, Z. Chen, Strong and weak convergence rates of logarithmic transformed truncated EM methods for SDEs with positive solutions, J. Comput. Appl. Math. 419 (2023) 114758. 
		
		\bibitem{M01}
		X. Mao, Stochastic Differential Equations and Applications, second ed., Academic Press (2008).
		
		%\bibitem{M11}
		%Mao, X.: Numerical solutions of stochastic differential delay equations under the generalized Khasminskii-type conditions. Appl. Math. Comput. \textbf{217}, 5512--5524 (2011)
		
		\bibitem{M15}
		X. Mao, The truncated Euler--Maruyama method for stochastic differential equations, J. Comput. Appl. Math. 290 (2015) 370--384.		
		
		\bibitem{M16}
		X. Mao, Convergence rates of the truncated Euler--Maruyama method for stochastic differential equations, J. Comput. Appl. Math. 296 (2016) 362--375. 
		
		\bibitem{MF21}
		X. Mao, F. Wei, T. Wiriyakraikul, Positivity preserving truncated Euler--Maruyama method for stochastic Lotka–Volterra competition model, J. Comput. Appl. Math. 394 (2021) 113566. 
		
		%\bibitem{SHMZ2013}
		%Song, M., Hu, L., Mao, X., Zhang, L.: Khasminskii-type theorems for stochastic functional differential equations. Discrete Contin. Dyn. Syst. Ser. B \textbf{18} (6), 1697--1714 (2013)
		
		%\bibitem{SS12}
		%Sabanis, S.: A note on tamed Euler approximations. Electron. Commun. Probab. \textbf{18}, 1--10 (2013) 
		
		%\bibitem{Gan01}
		%Wang, X., Gan, S.: The tamed Milstein method for commutative stochastic differential equations with non-globally Lipschitz continuous coefficients. J. Differ. Equ. Appl. \textbf{19}, 466--490 (2013) 
		
		\bibitem{TX2024}
		Y. Tang, X. Mao, The logarithmic truncated EM method with weaker conditions, Appl. Numer. Math. 198 (2024) 258--275. 
		
		\bibitem{NS2014}
		A. Neuenkirch, L. Szpruch, First order strong approximations of scalar SDEs defined in a domain, Numer. Math. 128 (1) (2014) 103--136.
		
		\bibitem{PPlogTM}
		Y. Yi, Y. Hu, J. Zhao, Positivity preserving logarithmic Euler–Maruyama type scheme for stochastic differential equations, Commun. Nonlinear Sci. Numer. Simul. 101 (2021) 105895.  
		
	\end{thebibliography}
	\end{document}